\documentclass[11pt,reqno]{amsart}
\usepackage[utf8]{inputenc}
\usepackage{amsmath,amssymb}
\usepackage{hyperref}
\usepackage{mathtools}
\usepackage{cite}
\usepackage{tikz}
\usetikzlibrary{automata}

\usepackage{mathrsfs}

\theoremstyle{plain}
\newtheorem{thm}{Theorem}[section]
\newtheorem{ddef}[thm]{Definition}
\newtheorem{cor}[thm]{Corollary}
\newtheorem{lem}[thm]{Lemma}
\newtheorem{prop}[thm]{Proposition}

\theoremstyle{remark}

\theoremstyle{definition}
\newtheorem{example}[thm]{Example}

\newcommand{\malcev}{\mathbin{\hbox{$\bigcirc$\rlap{\kern-8.25pt\raise0,50pt\hbox{${\tt
  m}$}}}}}
\newcommand{\smalcev}{\mathbin{\hbox{$\bigcirc$\rlap{\kern-7pt\raise0,30pt\hbox{${\tt
  m}$}}}}}

\newcommand{\HH}{\mathrel{\mathscr H}}

\newcommand{\soc}{\mathop{\mathrm{soc}}\nolimits}
\newcommand{\rad}{\mathop{\mathrm{rad}}\nolimits}

\usepackage{listings}
\lstdefinelanguage{GAP}{%
  morekeywords={%
    Assert,Info,IsBound,QUIT,%
    TryNextMethod,Unbind,and,break,%
    continue,do,elif,%
    else,end,false,fi,for,%
    function,if,in,local,%
    mod,not,od,or,%
    quit,rec,repeat,return,%
    then,true,until,while%
  },%
  sensitive,%
  morecomment=[l]\#,%
  morestring=[b]",%
  morestring=[b]',%
}[keywords,comments,strings]
\lstdefinestyle{sharpc}{language=[Sharp]C, frame=lr, rulecolor=\color{blue!80!black}}
\usepackage[T1]{fontenc}
\usepackage{xcolor}
\usepackage{bookmark}
\begin{document}
\title[Homogeneous Weights on Finite Semigroup Algebras]{Homogeneous Weights on Semigroup Algebras of Finite Commutative Semigroups}
\author{M.H. Shahzamanian}
\address{M.H. Shahzamanian\\ Centro de Matem\'atica e Departamento de Matem\'atica, Faculdade de Ci\^{e}ncias,
Universidade do Porto, Rua do Campo Alegre, 687, 4169-007 Porto, Portugal}
\email{m.h.shahzamanian@gmail.com}

\begin{abstract}
The determinant of a finite semigroup is closely related to the Frobenius property of its corresponding  semigroup algebra, and it plays a significant role in coding theory applications. 
In this paper, we study determinants of twisted contracted semigroup algebras over finite fields. 
We extend Steinberg's determinant theory to this setting by establishing a Frobenius criterion for twisted contracted semigroup algebras over finite fields. 
We also develop a structural theory for two-generated twisted contracted monoid algebras. 
In particular, we identify classes for which the determinant of the associated matrix depends only on its zero--nonzero pattern, and is therefore independent of the values of the cocycle. 
Consequently, it suffices to verify the determinant condition for the trivial character, from which the corresponding result for all characters follows. 
Under suitable assumptions, these semigroups are completely determined by a small set of defining identities. 
Motivated by the character-theoretic construction of homogeneous weights on finite Frobenius rings, we construct homogeneous weights for Frobenius twisted contracted semigroup algebras and investigate the non-Frobenius case.
In particular, we characterize the principal ideals on which the character-average weight satisfies the averaging property of homogeneous weights and thereby extend the classical notion of homogeneous weight to a prescribed family of principal ideals.
\end{abstract}

\maketitle

\section{Introduction}

The determinant of a finite group was introduced by Dedekind in the 1880s and
subsequently studied by Frobenius. At about the same time, Smith considered the
determinant of a \(G\times G\) matrix whose \((g,h)\)-entry is \(x_{gh}\), where
\(G\) is a finite group and \(x_k\) is an indeterminate associated with
\(k\in G\) \cite{Smith}. The notion of determinant was later extended from
groups to finite semigroups, and several aspects of the resulting semigroup
determinant have been investigated
\cite{Lindstr,Wilf,Wood}.

Besides its intrinsic algebraic interest, the semigroup determinant has important applications in coding theory. In particular, the extension of the MacWilliams theorem from finite fields to finite Frobenius rings is closely related to the non-vanishing of semigroup determinants \cite{Wood-Duality}. A major advance in the theory was made by Steinberg~\cite{Ste-Fac-det}, who developed a general theory of determinants of finite semigroups, established a close connection between semigroup determinants and Frobenius algebras, and obtained explicit factorizations and determinant formulas for several classes of semigroups, including finite commutative semigroups.
The relationship between the determinant of a finite-dimensional algebra and its Frobenius property is classical; see, for example, \cite[Chapter~16]{Okn}.

The computation of semigroup determinants has been investigated for several
families of finite semigroups. Building on Steinberg's general theory
\cite{Ste-Fac-det}, explicit determinant formulas have been obtained for
various classes of semigroups in
\cite{Sha-Det,Sha-Det2,Sha-Det3,Sha-Det4}. The present paper continues this line of
research by studying twisted contracted semigroup algebras over finite fields
and relating their determinants to homogeneous weights.

Homogeneous weights on finite Frobenius rings play a fundamental role in coding
theory. Introduced by Constantinescu and Heise for residue class rings and
subsequently generalized to arbitrary finite rings, they provide an important
tool in the study of linear codes over finite rings. Character-theoretic
descriptions of homogeneous weights have also been extensively investigated,
notably by Gluesing-Luerssen~\cite{Gluesing-Luerssen2016}.

The present paper is motivated by bringing together two lines of research,
namely semigroup determinants and homogeneous weights, in the setting of
twisted contracted semigroup algebras over finite fields. While
Steinberg's theory is developed over the complex field, our work is carried out
over finite fields. This not only requires adapting the determinant theory to
our setting, but also has the advantage that the resulting twisted contracted
semigroup algebras become finite rings. Consequently, they lie within the
framework of coding theory, allowing the study of homogeneous weights and
providing a foundation for future applications to linear codes over finite
rings.

Our first objective is to adapt Steinberg's determinant theory to twisted contracted semigroup algebras over finite fields. We establish the corresponding Frobenius criterion in this setting. 
We also investigate two-generated twisted contracted monoid algebras, identifying classes for which the Frobenius criterion is independent of the cocycle. We identify broad classes for which the determinant of the associated matrix depends only on its zero--nonzero pattern and is therefore independent of the cocycle defining the twisted semigroup algebra. Consequently, the Frobenius criterion reduces to verifying the determinant condition for the trivial character, from which the corresponding result for all characters follows automatically. Under suitable assumptions, we further show that these semigroups are completely determined by a small collection of defining identities.

Our second objective is to investigate homogeneous weights on these algebras via the character-average construction. We prove that, whenever the algebra is Frobenius, the resulting weight is a homogeneous weight in the classical sense. We then study the non-Frobenius case and show that, although condition {\rm(iii)} of the definition of a homogeneous weight does not hold on every nonzero principal ideal, it holds precisely on those principal ideals containing the minimal ideal generated by the zero element. This naturally leads to an extension of the classical notion of homogeneous weight to a prescribed family of principal ideals, which agrees with the classical definition in the Frobenius case.

The paper is organized as follows. In Section~\ref{Prelim}, we recall the
necessary background on based algebras, Frobenius algebras, semigroup
determinants, and homogeneous weights. In
Section~\ref{SemAlgebraComSem}, we adapt Steinberg's determinant results to
twisted contracted semigroup algebras over finite fields, establish the
corresponding Frobenius criterion, obtain explicit determinant formulas, and
construct the associated homogeneous weights in the Frobenius case. In
Section~\ref{Two-G-A}, we investigate two-generated twisted contracted monoid
algebras and obtain structural results that completely determine these
semigroups under suitable assumptions. Finally, in
Section~\ref{SecWeight}, we study the non-Frobenius case, characterize the
principal ideals on which the character-average weight satisfies
condition~{\rm(iii)}, and extend the classical notion of homogeneous weight to
a prescribed family of principal ideals.


\section{Preliminaries}\label{Prelim}
\subsection{Semigroups}

For standard notation and terminology regarding semigroups, we refer the reader to~\cite[Chap.~5]{Alm}, \cite[Chaps.~1--3]{Cli-Pre}, and~\cite[Appendix~A]{Rho-Ste}.

Let \( S \) be a finite semigroup, and let \( a, b \in S \). Green’s relations, originally introduced by Green~\cite{Gre}, are defined as follows:
\begin{itemize}
    \item \( a \mathrel{\mathcal{R}} b \) if and only if \( aS^1 = bS^1 \),
    \item \( a \mathrel{\mathcal{L}} b \) if and only if \( S^1a = S^1b \),
    \item \( a \mathrel{\mathcal{H}} b \) if and only if \( a \mathrel{\mathcal{R}} b \) and \( a \mathrel{\mathcal{L}} b \),
    \item \( a \mathrel{\mathcal{J}} b \) if and only if \( S^1aS^1 = S^1bS^1 \).
\end{itemize}
If \( S \) contains an identity element, then we set \( S^1 = S \). Otherwise, we define \( S^1 = S \cup \{1\} \), where \( 1 \) acts as an identity. 
The \( \mathcal{R}, \mathcal{L}, \mathcal{H}, \mathcal{J} \)-classes of an element \( a \in S \) are denoted by \( R_a, L_a, H_a \), and \( J_a \), respectively.

We also consider a refinement of the \( \mathcal{L} \)-relation, introduced by Fountain et al.~\cite{Fou-Gom-Gou}. We write \( a \mathrel{\widetilde{\mathcal{L}}} b \) if and only if \( a \) and \( b \) share the same set of idempotent right identities, that is,
\[
ae = a \quad \text{if and only if} \quad be = b.
\]
The dual relation \( \widetilde{\mathcal{R}} \) is defined analogously, and we set \( \widetilde{\mathcal{H}} = \widetilde{\mathcal{L}} \cap \widetilde{\mathcal{R}} \). The corresponding equivalence classes of an element \( s \in S \) are denoted by \( \widetilde{L}_s, \widetilde{R}_s \), and \( \widetilde{H}_s \), respectively. For additional background, see~\cite{Lawson-Sem-Cat}.

An element \( e \in S \) is called \emph{idempotent} if \( e^2 = e \). The set of all idempotents in \( S \) is denoted by \( E(S) \). 



\subsection{Frobenius algebras}
In modern terms, a finite-dimensional algebra $\Lambda$ over a field $K$ is called \emph{Frobenius} if there exists
a $K$-linear map $\lambda : \Lambda \to K$ such that $\ker(\lambda)$ contains no nonzero left or right ideal
\cite[Definition 1.6]{Benson-Rep-coh}.
A $K$-linear map $\lambda : \Lambda \to K$ satisfying this property is called a \emph{Frobenius form}.
The definition of a Frobenius algebra is equivalent to the existence of a left $\Lambda$-module
isomorphism
\[
{}_{\Lambda}\Lambda \;\cong\; \mathrm{Hom}_K(\Lambda_{\Lambda},K)
\]
\cite[Theorem~3.15]{LamTY}.
Equivalently, there is an isomorphism of right $\Lambda$-modules
\[
\Lambda_{\Lambda} \;\cong\; \mathrm{Hom}_K({}_{\Lambda}\Lambda,K),
\]
see, for example, \cite[Exercise~31.(5.3)]{MR1245487}.

The following proposition then easily follows:
\begin{prop}\label{DirectProLambda}
Let $R_1, \dots, R_t$ be finite-dimensional Frobenius $K$-algebras
with Frobenius forms $\lambda_1, \dots, \lambda_t$, respectively.
Then the direct product
\[
R = R_1 \times \cdots \times R_t
\]
is a Frobenius algebra with Frobenius form
\[
\lambda(a_1,\dots,a_t) = \sum_{i=1}^t \lambda_i(a_i),
\]
where each $\lambda_i$ is extended to $R$ via the projection onto the $i$-th component.
\end{prop}

%
%
%
%
%

A ring $R$ is called \emph{self-injective} if it is injective as a right regular $R_R$-module.
Following Lam \cite[Ch.~15--16]{LamTY}, a ring $R$ is called
\emph{quasi-Frobenius (QF)} if it is right noetherian and right self-injective.
An artinian QF ring $R$ is said to be \emph{Frobenius} if its socle
is isomorphic to its semisimple top as a right $R$-module, that is,
\[
\soc(R_R)_R \cong (R/\rad(R))_R,
\]
where $\soc(R_R)$ denotes the socle of the right $R$-modules $R$ and $\rad(R)$ is the Jacobson radical of $R$.
Equivalently, one may require
\[
{}_R\soc({}_R R) \cong {}_R(R/\rad(R)),
\]
since for a QF ring the left and right conditions are equivalent.
If $R$ is a finite-dimensional algebra over a field $K$, then
$R$ is a Frobenius ring if and only if it is a Frobenius $K$-algebra
\cite[Theorem~16.21]{LamTY}.  It shows that the
property of $R$ being a Frobenius $K$-algebra is actually independent of $K$
(as long as $R$ is finite-dimensional over $K$).

A character $\chi \in \widehat{R} = \operatorname{Hom}_{\mathbb{Z}}(R,\mathbb{C}^\times)$ is called \emph{generating} if $\widehat{R} = R\chi$, or equivalently, if $\ker \chi$ contains no nonzero left ideal of $R$; see, for example, \cite{Hon-Thom2001}.

Let $R$ be a finite Frobenius $K$-algebra with Frobenius form $\lambda : R \to K$, and let $\psi : K \to \mathbb{C}^\times$ be a nontrivial additive character. Since $R$ is a Frobenius ring, $\widehat{R}$ is a cyclic left $R$-module, and hence $R$ admits a generating character. In particular, there exists $u \in R$ such that the character
\[
\chi_u : R \to \mathbb{C}^\times, \qquad \chi_u(r) = \psi(\lambda(ur))
\]
is a generating character of $R$. 
Thus, a Frobenius form on $R$ naturally gives rise to a generating character of $R$.
While Frobenius algebras are classically studied via Frobenius forms, generating characters are fundamental in the study of finite Frobenius rings, particularly in coding theory. The above construction provides a bridge between these two frameworks.

The following proposition is useful for constructing generating characters of direct products of finite Frobenius rings.

\begin{prop}
Let 
\[
R = R_1 \times \cdots \times R_t,
\]
where each $R_i$ is a finite Frobenius ring. For each $1 \le i \le t$, let $\chi_i$ be a generating character of $R_i$. 

Then $R$ is a Frobenius ring and there is a natural isomorphism
\[
\widehat{R} \cong \widehat{R_1} \times \cdots \times \widehat{R_t}.
\]
Moreover, a generating character $\chi$ of $R$ is given by
\[
\chi(a_1,\dots,a_t)
=
\prod_{i=1}^t \chi_i(a_i),
\qquad (a_1,\dots,a_t) \in R.
\]

\end{prop}


\subsection{The determinant of a based algebra}
For standard notation and terminology relating to finite-dimensional algebras, the reader is referred to \cite{Assem-Ibrahim, Benson-Rep-coh}.

Let $K$ be a field. By a \emph{based algebra} we mean a finite-dimensional $K$-algebra $A$ together 
with a distinguished basis $B$. We often write $(A,B)$ for this pair. 

The multiplication on $A$ is determined by its structure constants with respect to $B$, defined by
\[
bb' = \sum_{b'' \in B} c_{b'',b,b'}\, b'',
\]
where $b,b' \in B$ and $c_{b'',b,b'} \in K$.

Let $X_B = \{x_b \mid b \in B\}$ be a set of variables in bijection with $B$. 
The \emph{Cayley table} of $(A,B)$ is the $B \times B$ matrix over $K[X_B]$ given by
\[
C(A,B)_{b,b'} = \sum_{b'' \in B} c_{b'',b,b'}\, x_{b''}.
\]

Matrices obtained by specializing $C(A,B)$ at elements of $K^B$ are called 
\emph{paratrophic matrices} (following Frobenius). Thus, $C(A,B)$ is the generic 
paratrophic matrix.

We define
\[
\theta_{(A,B)} = \det C(A,B)
\]
as an element of $K[X_B]$. Then $\theta_{(A,B)}$ is either identically zero 
or a homogeneous polynomial of degree $|B|$.
We say that $\theta_{(A,B)}$ is \emph{identically zero over $K$}, and write
\[
\theta_{(A,B)} \equiv 0 \quad \text{in } K[X_B],
\]
if it is the zero polynomial in $K[X_B]$, that is, if all of its coefficients vanish in $K$. 
Otherwise, we say that $\theta_{(A,B)}$ is \emph{not identically zero}.

Let $S$ be a finite semigroup. The semigroup $K$-algebra $KS$ consists of all formal sums
\[
\sum_{s \in S} \lambda_s s,
\]
where $\lambda_s \in K$. Addition is defined componentwise, and multiplication is induced by the
semigroup operation on $S$, namely
\[
\left( \sum_{s \in S} \lambda_s s \right)
\left( \sum_{t \in S} \mu_t t \right)
=
\sum_{s,t \in S} \lambda_s \mu_t (st).
\]
Then $KS$ is a finite-dimensional $K$-algebra with basis $S$.

Let $X_S = \{ x_s \mid s \in S \}$ be a set of variables in bijection with $S$.
The multiplication of $S$ can be encoded in the \emph{Cayley matrix}
\[
C(S) = \big( x_{st} \big)_{s,t \in S},
\]
which is an $|S| \times |S|$ matrix with entries in the polynomial ring $K[X_S]$.

The polynomial $\theta_S$, in the case $K = \mathbb{C}$, is called the 
\emph{semigroup determinant} (or \emph{Dedekind--Frobenius semigroup determinant}) of $S$. 
We shall continue to refer to $\theta_S$ as the \emph{semigroup determinant} for an arbitrary fixed field $K$.
For more details, see \cite{Frobenius1903theorie}, \cite[Chapter~16]{Okn}, and \cite{Ste-Fac-det}.

As we mentioned on the previous subsection, a finite-dimensional unital $K$-algebra $A$ is called \emph{Frobenius} if there exists 
a linear map $\lambda : A \to K$ (called a Frobenius form) such that the bilinear form
\[
(a,b) \mapsto \lambda(ab)
\]
is nondegenerate. Equivalently, $\ker \lambda$ contains no nonzero left or right ideal of $A$.

By \cite[Theorem~2.1]{Ste-Fac-det}, the semigroup determinant $\theta_{(A,B)}$ is nonzero
if and only if the based algebra $(A,B)$ is a Frobenius algebra, for the case $K=\mathbb{C}$.
This result may be viewed as an instance of Frobenius’s criterion
(cf.\ \cite[§16.82]{LamTY}), which predates the modern definition of Frobenius algebras.
Theorem~2.1 of \cite{Ste-Fac-det} also holds over any infinite field $K$.
For finite fields, one must ensure that the nonzero polynomial $\theta_{(A,B)}$
admits a nonzero specialization over $K$; this is guaranteed, for example, if
$|K| > |B|$.

\begin{thm}\label{DetAB}
Let $(A,B)$ be a based algebra over a field $K$. Assume that either $K$ is infinite or $|K| > |B|$. Then
\[
\theta_{(A,B)} \not\equiv 0
\]
if and only if $A$ is a Frobenius algebra (and, in particular, unital).
\end{thm}

\begin{proof}
The proof follows the same lines as that of \cite[Theorem~2.1]{Ste-Fac-det} in the case of infinite fields. 

For finite fields, the only additional point is to ensure the existence of a specialization 
$\lambda : B \to K$ such that $\theta_{(A,B)}(\lambda) \neq 0$. 
This is guaranteed by the following lemma.
\begin{lem}\label{OrderKAndx1xn}
Let $K$ be a finite field with $|K| = q$, and let $f \in K[x_1,\dots,x_n]$ be a nonzero polynomial of total degree $d$. 
If $q > d$, then there exists $a \in K^n$ such that $f(a)\neq 0$.
\end{lem}
Since $\theta_{(A,B)}$ is a nonzero polynomial of degree $|B|$, the condition $|K| > |B|$ ensures the existence of such a specialization. 
The remainder of the argument is identical.
\end{proof}


Suppose that $S$ has a zero element $\{z\}$.
Let \( \widetilde{X} = X_{S \setminus \{z\}} \) if \( S \) is understood.  
Proposition~2.7 of \cite{Ste-Fac-det} (see also \cite{Wood}) extends verbatim to an arbitrary field $K$, as its proof relies only on formal manipulations of structure constants.
In particular, it establishes a connection between the contracted semigroup determinant and the semigroup determinant of a semigroup $S$ with a zero element.
There is a $K$-algebra isomorphism
\[
KS \cong K_0S \times K z,
\]
sending $s \in S \setminus \{z\}$ to $(s,0)$ and $z$ to $(0,z)$.
Put \( y_s = x_s - x_z \) for \( s \neq z \) and let \( Y = \{ y_s \mid s \in S \setminus \{z\} \} \).  
Then  
\[
\theta_S(X) = x_z \widetilde{\theta}_S(Y).
\]  





For a semigroup $S$, a \emph{twisted semigroup algebra}
of $S$ over a field $K$ is a $K$-algebra with basis $\{\overline{s} \mid s \in S\}$ such that
\[
\overline{s} \cdot \overline{t} = c(s,t)\, \overline{st},
\]
for some map $c : S \times S \to K^{\times}$.
The map $c$ is a $2$-cocycle, and we denote this algebra by $K(S,c)$.
%
%
If $S$ is a semigroup with zero $z$, a \emph{twisted contracted semigroup algebra}
of $S$ over $K$ is a $K$-algebra with basis $\{s \mid \overline{s} \in S \setminus \{z\}\}$
such that
\[
\overline{s} \cdot \overline{t} =
\begin{cases}
c(s,t)\, \overline{st} & \text{if } st \neq z,\\
0 & \text{if } st = z,
\end{cases}
\]
where $c(s,t) \in K^{\times}$ whenever $st \neq z$.
We denote this algebra by $K_0(S,c)$.


\subsection{Homogeneous Weight}\label{Hom-Wei}

Constantinescu and Heise~\cite{MR1476368} introduced
homogeneous weights on integer residue rings, characterized by the
properties that they are constant on classes of associated elements
and that the total weight of every nonzero ideal is proportional to
its cardinality. The weight was subsequently generalized to arbitrary
non-commutative finite rings by Honold and
Nechaev~\cite{Nec-Hon-Thom2001} and Greferath and
Schmidt~\cite{Gre-Sch-Hom-Wei}.
Honold~\cite{Hon-Thom2001} provides an explicit formula for the values
of the homogeneous weight on Frobenius rings.
In particular, Gluesing-Luerssen~\cite{Gluesing-Luerssen2016}
studies the homogeneous weight for finite Frobenius rings that are
isomorphic to a product of local rings.

\begin{ddef}\label{homogeneous-weight}
Let $R$ be a finite Frobenius ring.
The \emph{(left) homogeneous weight} on $R$ with average value $\gamma$
is a function
\[
\omega : R \longrightarrow \mathbb{R}
\]
such that
\begin{enumerate}
\item[(i)] $\omega(0)=0$;

\item[(ii)] $\omega(x)=\omega(y)$ for all $x,y\in R$ such that
\[
Rx = Ry;
\]

\item[(iii)] for all $x \in R \setminus \{0\}$,
\[
\sum_{y \in Rx} \omega(y) = \gamma \, |Rx|,
\]
that is, the average weight over each nonzero principal ideal is $\gamma$.
\end{enumerate}
\end{ddef}


We consider the homogeneous weight with
average value $\gamma = 1$, which we refer to as the
\emph{normalized homogeneous weight}. Following Honold~\cite[p.~412]{Hon-Thom2001},
for a finite Frobenius ring $R$ with generating character $\chi$, the
normalized homogeneous weight is explicitly given by
\begin{equation}\label{NorHomWeight}
	\omega(r)
	=
	1 - \frac{1}{|R^\times|} \sum_{u \in R^\times} \chi(ru),
	\qquad r \in R,
\end{equation}
where $R^\times$ is the group of units of $R$.

In \cite{Gluesing-Luerssen2016}, the normalized homogeneous weight
is described for a direct product of Frobenius rings in general, and
more specifically for a ring that is a direct product of finite local
Frobenius rings. These results are stated respectively in the following
propositions.

\begin{prop}\label{WeightDirectFrobenius} \cite[Proposition~3.7]{Gluesing-Luerssen2016}
Let
\[
R = R_1 \times \cdots \times R_t
\]
be a direct product of Frobenius rings.
For each $1\leq i \in t$, let $\omega_i$ be the normalized homogeneous
weight on $R_i$. Then the normalized homogeneous weight on $R$ is
given by
\[
\omega(a_1,\dots,a_t)
=
1 - \prod_{i=1}^t \bigl(1 - \omega_i(a_i)\bigr),
\qquad
(a_1,\dots,a_t) \in R.
\]
\end{prop}

\begin{prop}\label{WeightDirectFrobeniusLocal} \cite[Proposition~3.8]{Gluesing-Luerssen2016}
Let
\[
R = R_1 \times \cdots \times R_t,
\]
where each $R_i$ is a finite local Frobenius ring with residue field
$R_i/\rad(R_i)$ of order $q$. Then the homogeneous weight
on $R$ is given by
\[
\omega(a)
=
\begin{cases}
\displaystyle
1 - \left(\frac{-1}{q-1}\right)^{\operatorname{wt}(a)},
& \text{if } a \in \soc(R), \\[1.2ex]
1, & \text{otherwise},
\end{cases}
\]
where $\operatorname{wt}(a)=|\{i | a_i \neq 0\}|$ denotes the Hamming weight of $a = (a_1,\dots,a_t)$.
\end{prop}

In \cite{Gluesing-Luerssen2016}, let $R$ be a finite Frobenius ring of the form
\begin{equation}\label{form1}
R = R_1 \times \cdots \times R_t,
\end{equation}
where each
\[
R_i = R_{i,1} \times \cdots \times R_{i,n_i}
\]
is a finite product of local Frobenius rings $R_{i,j}$ satisfying
\[
|R_{i,j}/\rad(R_{i,j})| = |\soc(R_{i,j})| = q_i
\quad \text{for all } 1\leq j \leq n_i,
\]
with $q_1,\dots,q_t$ distinct.
Recall also that
\[
\soc(R_{i,j}) \cong R_{i,j}/\rad(R_{i,j}).
\]
For such $R$, the following theorem is proved  \cite[Theorem 3.9]{Gluesing-Luerssen2016}. 

\begin{thm}\label{WeightDirectFrobeniusLocalDif}
Let $R$ be as in \eqref{form1} and write its elements as
$a = (a_1,\dots,a_t)$, where $a_i \in R_i$. Using the Hamming weight
$\operatorname{wt}$ on each $R_i$, the homogeneous weight
on $R$ is given by
\[
\omega(a_1,\dots,a_t)
=
\begin{cases}
\displaystyle
1 - \prod_{i=1}^t \left(\frac{-1}{q_i - 1}\right)^{\operatorname{wt}(a_i)}, 
& \text{if } a \in \soc(R), \\[1.2ex]
1, & \text{otherwise}.
\end{cases}
\]
\end{thm}


\section{Semigroup Algebras of Finite Commutative Semigroups over Admissible Fields and Their Homogeneous Weights}\label{SemAlgebraComSem}

In this section, we revisit results of Steinberg (see Sections~5 and~6 of \cite{Ste-Fac-det}) concerning the structure of semigroup algebras of finite commutative semigroups. While these results are originally formulated over the complex field $\mathbb{C}$, our aim is to adapt them to a suitably chosen finite field $K$.
More precisely, for a given finite commutative semigroup $S$, we construct a finite field $K$ such that the structural results remain valid over $K$, in the same spirit as Theorem~\ref{DetAB} from the preliminary section.
Furthermore, when $KS$ is a Frobenius algebra, we combine these results with those of Subsection~\ref{Hom-Wei} to obtain the homogeneous weight on $KS$.

We begin by recalling the relevant results from Section~5 of Steinberg’s paper, and then proceed to those from Section~6.


A semigroup $S$ with zero element $z$ is called \emph{nilpotent} if there exists a positive integer $k$ such that
\[
S^k = \{z\}
\quad \text{and} \quad
S^{k-1} \neq \{z\}.
\]
The integer $k$ is called the \emph{nilpotency degree} of $S$.

Then, for every nonzero element $s \in S \setminus \{z\}$, there exists an integer $1<k' \le k$ such that
\[
s^{k'} = z \quad \text{and} \quad s^{k'-1} \neq z.
\]
It follows that the elements $s^i$, for $1 \le i \le k'-1$, are pairwise distinct. Otherwise, if $s^{i_1} = s^{i_2}$ for some integers $i_1 \neq i_2$ with $i_1, i_2 < k'$, then
\[
s^{k'-1} = s^{i_1} s^{k'-1-i_1} = s^{i_2} s^{k'-1-i_1} = s^{k''},
\]
for some $k'' > k'-1$, which contradicts $s^{k''} = z$ and $s^{k'-1} \neq z$. We say that $k'$ is the \emph{nilpotency degree} of the element $s$.

Let $S$ be a nilpotent semigroup and let $M = S \cup \{1\}$ be the monoid obtained by adjoining an identity element $1$. 
Let $K_0(M,c)$ be a twisted contracted monoid algebra over a field $K$. Then $K_0(M,c)$ is a local ring, and $K_0(S,c)$ is its Jacobson radical. 
If, moreover, $K_0(M,c)$ is Frobenius, then its socle is simple and one-dimensional over $K$.

A nonzero element $m \in M$ is called
 \emph{left} (respectively,  \emph{right}) \emph{annihilating} if 
\(
mS = \{z\} (\text{respectively}, Sm =\{z\}).
\)
If $m$ is both left and right annihilating, we say that $m$ is \emph{annihilating}.

By \cite[Proposition~5.1]{Ste-Fac-det}, if $K_0(M,c)$ is Frobenius, then $M$ has a unique right annihilating element $z'$. Moreover, $z'$ is annihilating and is also the unique left annihilating element of $M$. In particular, the left and right socles of $K_0(M,c)$ coincide and are equal to $K z'$. Although the original theorem is stated over $\mathbb{C}$, the result holds over any field $K$ by the same argument.

By \cite[Theorem~5.2]{Ste-Fac-det}, Steinberg computes the determinant 
$\widetilde{\theta}_{M,c}$ of $K_0(M,c)$ in the case $K=\mathbb{C}$, showing 
that it vanishes identically unless $M$ has a unique annihilating element $z'$. 
This result extends to an arbitrary field $K$. In the case where $K$ is finite, 
we assume in addition that $|K| > |S|$; under this condition, the proof is identical 
to that of Steinberg.
In this case, define an $M \setminus \{z\} \times M \setminus \{z\}$ matrix $A$ by
\[
A_{s,t} =
\begin{cases}
c(s,t), & \text{if } st = z', \\
0, & \text{otherwise}.
\end{cases}
\]
Then
\[
\widetilde{\theta}_{M,c} = \det(A)\, x_{z'}^{\,|S|},
\]
where $\det(A)$ is the determinant of $A$ computed over the field $K$.

In fact, the proof of this theorem yields a stronger result. When $K_0(M,c)$ is a Frobenius algebra, it provides an explicit Frobenius form, which we record in the following proposition.

\begin{prop}\label{GenCharMc}
Suppose that $K_0(M,c)$ is a Frobenius algebra. If $K$ is finite, assume in addition that $|K| > |S|$. Then a Frobenius form of $K_0(M,c)$ is given by the linear extension of the map $\lambda : M \to K$ defined by
\[
\lambda(z') = 1 \quad \text{and} \quad \lambda(s) = 0 \text{ for all } s \neq z',
\]

Moreover, viewing $K_0(M,c)$ as a finite Frobenius ring, it admits a generating character. More precisely, if $\psi : K \to \mathbb{C}^\times$ is a nontrivial additive character with $\psi(1)\neq 1$, then 
\[
\chi=\psi\circ\lambda
\]
is a generating character of $K_0(M,c)$.
\end{prop}

\begin{proof}
Let $I$ be a left ideal of $R=K_0(M,c)$ such that $I \subseteq \ker(\chi)$. For $x=\sum_{s\in M}\alpha_s s \in I$, we have $\lambda(x)=\alpha_{z'}$, and hence $\psi(\alpha_{z'})=1$. If $\alpha_{z'} \neq 0$, then $\alpha_{z'}^{-1}x \in I$ and $\psi(1)\neq 1$, a contradiction. Thus $\alpha_{z'}=0$ for all $x \in I$, so $I \subseteq \ker(\lambda)$. Since $\lambda$ is a Frobenius form, $\ker(\lambda)$ contains no nonzero left ideal, and hence $I=0$. Therefore $\chi$ is generating.
\end{proof}


Now, we move to section 6 of Steinberg’ paper.
Let $S$ be a finite semigroup such that $S^2 = S$ and $E(S)$ is contained
in the center of $S$, that is, for all $e \in E(S)$ and $s \in S$, we have
\(es = se\).

Since $E(S)$ is commutative, it is a meet-semilattice with respect to the ordering \(e \leq f\) if $ef = e$.
For $s \in S$, the set
\[
\{\, e \in E(S) \mid se = s \,\}
\]
is a nonempty subsemigroup of $E(S)$ and hence possesses a unique
minimal element, which we denote by $s^{+}$.
Note that $e^{+} = e$ for all $e \in E(S)$.
We define an equivalence relation on $S$ by
\[
s \sim t
\quad \text{if and only if} \quad
s^{+} = t^{+}.
\]
Let $\widetilde{H}_s$ denote the equivalence class of $s \in S$ (this is the $\widetilde{\HH}$-class of $s$ in the sense of Fountain et al.~\cite{Fou-Gom-Gou}).
The idempotent $s^{+}$ is the unique idempotent contained in $\widetilde{H}_s$.

Let $e \in E(S)$. The maximal subgroup of $S$ at $e$,
denoted by $G_e$, is the group of units of $eSe$.
Define
\[
I_e = \{\, s \in S \mid s^{+} < e \,\}.
\]
Then $I_e$ is an ideal of $S$, and $I_e = \emptyset$ if and only if
\[
eSe = G_e = \widetilde{H}_e.
\]

We define
\[
\widetilde{H}_e^{0} =
\begin{cases}
eSe/I_e, & \text{if } I_e \neq \emptyset, \\[0.8ex]
G_e \cup \{z\}, & \text{if } I_e = \emptyset,
\end{cases}
\]
where $z$ is an adjoined zero (note that $\widetilde{H}_e = G_e$
when $I_e = \emptyset$).

Then $\widetilde{H}_e^{0}$ is a monoid with identity $e$ and may be
identified with $\widetilde{H}_e \cup \{z\}$, where $z$ is a zero
element. For $a,b \in \widetilde{H}_e$, the multiplication is given by
\[
a \cdot b =
\begin{cases}
ab, & \text{if } ab \in \widetilde{H}_e, \\[0.6ex]
z, & \text{otherwise}.
\end{cases}
\]
As no confusion will arise, we simply write $ab$ instead of
$a \cdot b$.

The group of units of $\widetilde{H}_e^{0}$ is $G_e$, and every
nonunit element of $\widetilde{H}_e^{0}$ is nilpotent. 
By \cite[Theorem 6.5]{Ste-Fac-det}, the mapping
\begin{equation}\label{Thm6.5}
Z : KS \longrightarrow \prod_{e \in E(S)}K_0 \widetilde{H}_e^{\,0}
\end{equation}
given by \(Z(s) = \sum_{t \le s} t\ (s \in S)\)
is an isomorphism of $K$-algebras. In particular, $KS$ is unital.


In (\ref{Thm6.5}), the algebra $KS$ is isomorphic to a direct product of complex algebras of certain monoids with zero, where the submonoid obtained by removing the group of units is nilpotent. We now present Steinberg’s results in detail the structure of these complex monoid algebras in the commutative case.

Let $M$ be a commutative monoid with zero $z$ and group of units $G$ such that $M\setminus G$ consists of nilpotent elements.
Then, let
\[
M/G = \{\, Gm \mid m \in M \,\}
\]
is a monoid with multiplication defined by
\(
Gm_1 \cdot Gm_2 = Gm_1 m_2,
\)
and it is a quotient of $M$ via the canonical homomorphism
\(
m \mapsto Gm.
\)

Since $G$ is a finite abelian group, all irreducible representations of $G$ are one-dimensional and hence correspond to group homomorphisms
\[
\chi : G \to \mathbb{C}^\times.
\]
In particular, every such character takes values in the group of $n$-th roots of unity, where $n = |G|$.

Choose a prime $p$ such that $p \nmid n$. Since $(\mathbb{Z}/n\mathbb{Z})^\times$ is finite, there exists an integer $k \ge 1$ such that
\[
p^k \equiv 1 \pmod n.
\]
Hence $n \mid (p^k - 1)$, and the finite field $\mathbb{F}_q$ with $q=p^k$ contains all $n$-th roots of unity, since $\mathbb{F}_q^\times$ is cyclic of order $q-1$.
It follows that every character $\chi : G \to \mathbb{F}_q^\times$ is well-defined, and hence all irreducible representations of $G$ can be realized over $\mathbb{F}_q$. Moreover, for every $i \ge 1$, the extension field $\mathbb{F}_{q^i}$ also contains all $n$-th roots of unity and therefore serves as a splitting field for $G$.

Choosing $i$ sufficiently large so that $|\mathbb{F}_{q^i}| > |M|$, we set $K=\mathbb{F}_{q^i}$. We denote by $\widehat{G}$ the group of characters
\[
\chi : G \to K^\times.
\]
In this way, all irreducible representations of $G$ are one-dimensional and their values are roots of unity of order dividing $n$. Since $\operatorname{char}(K)=p$ and $p \nmid n$, we also have $\operatorname{char}(K)\nmid |G|.$

For $\chi \in \widehat{G}$, define
\[
e_{\chi}
=
\frac{1}{|G|}
\sum_{g \in G} \chi(g)^{-1}\, g
\in KG.
\]

For $\chi \in \widehat{G}$, define
\[
e_{\chi}
=
\frac{1}{|G|}
\sum_{g \in G} \chi(g)^{-1}\, g
\in KG.
\]

Since $\operatorname{char}(K)\nmid |G|$, the elements $e_\chi$ are well-defined idempotents in $KG$, and we have
\[
1 = \sum_{\chi \in \widehat{G}} e_{\chi},
\]
which gives a decomposition of $1$ into orthogonal primitive idempotents of $KG$.

Since $KG \subseteq K_0 M$ and $K_0 M$ is commutative, we obtain an isomorphism of $K$-algebras
\begin{equation}\label{6.1}
K_0 M
\cong
\prod_{\chi \in \widehat{G}} e_{\chi} K_0 M e_{\chi}
=
\prod_{\chi \in \widehat{G}} e_{\chi} K_0 M.
\end{equation}

We have $Mm =Mm'$ if and only if $Gm =Gm'$. Then, fix, once and for all, representatives $m_1,\dots,m_r$ of the
$G$-orbits on $M \setminus \{z\}$, and without loss of generality
assume that $m_1 = 1$.
If $m_i m_j \neq z$, define $f(i,j)$ by the condition
\[
m_i m_j \in G m_{f(i,j)}.
\]
Let $G_i$ denote the stabilizer of $m_i$ in $G$. 

Let $\chi \in \widehat{G}$, and define
\[
I_{\chi} = \{\, m \in M \mid \text{the stabilizer of $m$ in $G$ is not contained in } \ker \chi \,\}.
\]
Then $I_{\chi} = \emptyset$ if $\chi$ is the trivial character $1_G$, 
and otherwise $I_{\chi}$ is a proper ideal of $M$. 
Define
\[
M_{\chi} =
\begin{cases}
M, & \text{if } \chi = 1_G,\\
M/I_{\chi}, & \text{otherwise.}
\end{cases}
\]
The ideal $I_{\chi}$ is compatible with the equivalence relation
defining $M/G$. Hence the quotient $M_{\chi}/G$ is well defined.
Note that $M_{\chi}/G$ is obtained by adjoining an identity to a
nilpotent commutative semigroup.
If $m \in M \setminus I_{\chi}$, we may define
\[
\chi(m) = \chi(g) \quad \text{where } m = g m_i \text{ with } g \in G
\]
This is well-defined: if $m = h m_i$ as well, then $h^{-1}g \in G_i \subseteq \ker \chi$, and hence $\chi(g)=\chi(h)$.
Let
\(
J_{\chi} = \{\, i \mid G_i \subseteq \ker \chi \,\}.
\)

The proof of \cite[Proposition~6.9]{Ste-Fac-det} remains valid over $K$, since it relies only on:
the semisimplicity of $KG$ (guaranteed by $\operatorname{char}(K)\nmid |G|$),
the existence of linear characters $\chi : G \to K^\times$ (as $K$ is a splitting field for $G$),
and formal manipulations of structure constants.

Then, we have an isomorphism of \emph{twisted} algebras
\begin{equation}\label{twisted}
e_{\chi} K_0 M \cong K_0(M_{\chi}/G, c_{\chi}),
\end{equation}
where the twisting 2-cocycle \(c_{\chi}\) is given by
\[
c_{\chi}(G m_i, G m_j) = \chi(m_i m_j) \quad \text{if } m_i m_j \notin I_{\chi}.
\]
In particular, if \(m_i m_j = m_{f(i,j)}\) whenever \(m_i m_j \notin I_{\chi}\), then the twist is trivial and we have the untwisted algebra
\[
e_{\chi} K_0 M \cong K_0[M_{\chi}/G].
\]

Then, the contracted semigroup determinant
\(
\widetilde{\theta}_{M_{\chi}/G, c_{\chi}}
\)
vanishes identically unless $M_{\chi}/G$ has a unique annihilating element $G m_{i_{\chi}}$.
In this case,
\[
\widetilde{\theta}_{M_{\chi}/G, c_{\chi}} 
= \det A(\chi) \cdot x_{G m_{i_{\chi}}}^{\,|M_{\chi}/G|-1},
\]
where $A(\chi)$ is the $J_{\chi} \times J_{\chi}$ matrix with entries
\[
A(\chi)_{ij} =
\begin{cases}
\chi(m_i m_j), & \text{if } m_i m_j \in G m_{i_{\chi}}, \\
0, & \text{otherwise},
\end{cases}
\quad i,j \in J_{\chi}.
\]

Now let $S$ be a finite commutative semigroup. 
%
%
%
%

\begin{ddef}
Let $S$ be a finite commutative semigroup. 
A finite field $K$ is called \emph{admissible for $S$} if:
\begin{enumerate}
\item $\operatorname{char}(K)\nmid |G_e|$ for all $e \in E(S)$,
\item $K$ is a splitting field for each maximal subgroup $G_e$ of $S$,
\item $|K| > |S|$.
\end{enumerate}
\end{ddef}

We fix an admissible field $K$ for $S$.

Such a field can be constructed by taking 
\[
N = \operatorname{lcm}\{\, |G_e| \mid e \in E(S)\,\},
\]
choosing a prime $p \nmid N$, and then a power $q=p^k$ with $q \equiv 1 \pmod N$, so that $\mathbb{F}_q$ contains all $N$-th roots of unity. Taking a sufficiently large extension $\mathbb{F}_{q^i}$ yields the desired field $K$.

By \cite[Theorem~6.6]{Ste-Fac-det}, the semigroup determinant of $S$ is equal to the product of the determinants of $\widetilde{H}_e^{0}$ appearing in \eqref{Thm6.5}, after a suitable substitution of variables.
Hence, the semigroup determinant of $S$ is not identically zero (equivalently, $KS$ is a Frobenius algebra; see \cite[Theorem~2.1]{Ste-Fac-det}) if and only if, for each $e \in E(S)$, the determinant of $\widetilde{H}_e^{0}$ is not identically zero.
By the preceding analysis, this holds if and only if, for each $e \in E(S)$ and each $\chi \in \widehat{G_e}$, the semigroup $(\widetilde{H}_e^0 / I_{\chi}) / G_e$ has a unique annihilating element and $\det A(\chi) \neq 0$, where the determinant is computed in the field $K$.

%
%
%

Moreover, by \eqref{Thm6.5}, \eqref{6.1}, and \cite[Proposition~6.9]{Ste-Fac-det}, in the case where $S$ is commutative, the algebra $KS$ is isomorphic to a finite direct product of local rings. 
Each factor is a \emph{twisted contracted monoid algebra} of a commutative nilpotent semigroup with an adjoined identity, as observed in \cite[Remark~6.15]{Ste-Fac-det}.

We provide an explicit description of a Frobenius form of $KS$ in the case where it is a Frobenius algebra.  
Combining \eqref{Thm6.5}, \eqref{6.1}, and \eqref{twisted}, we obtain an isomorphism of $K$-algebras
\begin{equation}\label{isoCS2}
KS 
\cong \prod_{e \in E(S)} K_0 \widetilde{H}_e^{\,0}
\cong \prod_{e \in E(S)} \prod_{\chi \in \widehat{G_e}} e_{\chi} K_0 \widetilde{H}_e^{\,0}
\cong \prod_{e \in E(S)} \prod_{\chi \in \widehat{G_e}} K_0\bigl((\widetilde{H}_e^{\,0})_{\chi}/G_e,\, c_{(e,\chi)}\bigr).
\end{equation}

By Propositions~\ref{DirectProLambda} and~\ref{GenCharMc}, we obtain the following result.

\begin{thm}\label{GenCharS}
Let $S$ be a finite commutative semigroup, and let $K$ be an admissible field  for $S$. Suppose that $KS$ is a Frobenius algebra.  
Then a Frobenius form $\lambda : KS \to K$ is given componentwise by
\[
\lambda
=
\sum_{e \in E(S)} \sum_{\chi \in \widehat{G_e}} \lambda_{e,\chi},
\]
where each
\[
\lambda_{e,\chi} : K_0\bigl((\widetilde{H}_e^{\,0})_{\chi}/G_e,\, c_{(e,\chi)}\bigr) \longrightarrow K
\]
is a Frobenius form of the corresponding twisted contracted monoid algebra, as described in Proposition~\ref{GenCharMc}, defined by
\[
\lambda_{e,\chi}(z_e) = 1 
\quad \text{and} \quad 
\lambda_{e,\chi}(s) = 0 \ \text{for all } s \neq z_e,
\]
where $z_e$ is the unique annihilating element of $(\widetilde{H}_e^{\,0})_{\chi}/G_e$.

Moreover, $KS$ admits a generating character
\[
\kappa
=
\prod_{e \in E(S)} \prod_{\chi \in \widehat{G_e}} \chi_{e,\chi},
\]
where 
\[
\chi_{e,\chi} = \psi \circ \lambda_{e,\chi},
\]
and $\psi : K \to \mathbb{C}^\times$ is a nontrivial additive character with $\psi(1)\neq 1$. 
\end{thm}


We are now in a position to apply the results of Subsection~\ref{Hom-Wei} together with Theorem~\ref{GenCharS} to construct homogeneous weights on semigroup algebras of finite commutative semigroups.

Let $S$ be a finite commutative semigroup, and let $K$ be an admissible field for $S$ such that $KS$ is a Frobenius algebra. Then the homogeneous weight on $KS$ is given as follows.

Under the isomorphism \eqref{isoCS2}, each $x \in KS$ can be written as
\[
x = \bigl(x_{(e,\chi)}\bigr)_{e,\chi}.
\]

By~(\ref{NorHomWeight}) and Theorem~\ref{GenCharS}, the normalized homogeneous weight on $KS$ is given by
\[
\omega(x)
=
1 - \frac{1}{|(KS)^\times|} \sum_{u \in (KS)^\times} \kappa(xu),
\]
and, by Proposition~\ref{WeightDirectFrobenius},
\[
\omega(x)
=
1 - \prod_{e \in E(S)} \prod_{\chi \in \widehat{G_e}} 
\bigl(1 - \omega_{(e,\chi)}(x_{(e,\chi)})\bigr),
\]
where
\[
\omega_{(e,\chi)}(r)
=
1 - \frac{1}{|R_{e,\chi}^\times|} \sum_{u \in R_{e,\chi}^\times} \chi_{e,\chi}(ru),
\quad r \in R_{e,\chi},
\]
and we write
\[
R_{e,\chi} = K_0\bigl((\widetilde{H}_e^{\,0})_{\chi}/G_e,\, c_{(e,\chi)}\bigr).
\]

Thus, computing $\omega$ reduces to computing each $\omega_{(e,\chi)}$.

Since
\[
KS \cong \prod_{e \in E(S)} \prod_{\chi \in \widehat{G_e}} R_{e,\chi},
\]
where each $R_{e,\chi}$ is a finite local Frobenius ring with residue field isomorphic to $K$, Proposition~\ref{WeightDirectFrobeniusLocal} applies with $q = |K|$.

Therefore, the homogeneous weight on $KS$ is given by
\[
\omega(a)
=
\begin{cases}
\displaystyle
1 - \left(\frac{-1}{q-1}\right)^{\operatorname{wt}(a)},
& \text{if}\ a \in \soc(KS), \\[1.2ex]
1, & \text{otherwise},
\end{cases}
\]
where $a = (a_{(e,\chi)})_{e,\chi}$ under the above decomposition, and
\[
\operatorname{wt}(a)
=
\bigl|\{(e,\chi) \mid a_{(e,\chi)} \in K^\times z_{e}\}\bigr|,
\]
that is, the number of components in which $a_{(e,\chi)}$ lies in the nonzero socle of $R_{e,\chi}$.

If different products are defined in~(\ref{isoCS2}) using different admissible fields, then their corresponding homogeneous weights can also be computed by Theorem~\ref{WeightDirectFrobeniusLocalDif}.


\section{Two-Generated Twisted Contracted Monoid Algebras}\label{Two-G-A}


As mentioned in the previous section, to determine whether the semigroup algebra of a finite commutative semigroup over an admissible finite field is a Frobenius algebra, it suffices to consider the case where $M^{(G)}$ is a finite commutative monoid with zero $z$ and group of units $G$ such that $M^{(G)}\setminus G$ consists of nilpotent elements. 
Then, to determine whether the monoid algebra $K_0M^{(G)}$ over an admissible finite field $K$ is Frobenius, it is necessary and sufficient to verify that, for every character $\chi:G\to K^\times$, the algebra
\[
K_0(M^{(G)}_{\chi}/G,c_{\chi})
\]
is Frobenius. Equivalently, each matrix $A(\chi)$ has nonzero determinant.

Throughout this section, let $S$ be a finite commutative nilpotent semigroup with zero $z$, and let
\[
M=S\cup\{1\}.
\]
Let $B$ be a finite generating set of $S$. We study the twisted contracted monoid algebra
\[
R=K_0(M,c)
\]
over an admissible finite field $K$.

The results of this section apply, in particular, to each quotient monoid $M^{(G)}_{\chi}/G$ by replacing $M$ with $M^{(G)}_{\chi}/G$ and the cocycle $c$ with $c_{\chi}$.

A principal objective of this section is to identify classes of two-generated monoids for which the Frobenius criterion is independent of the particular cocycle. For these classes, it suffices to verify the determinant condition for the trivial character, from which the corresponding result for all characters follows automatically. We then investigate the remaining cases by developing structural properties of two-generated monoids, leading to a classification under suitable assumptions.

By the results of Section~\ref{SemAlgebraComSem}, if $R$ is Frobenius, then $M$ has a unique annihilating element $z'$, and the associated matrix
\[
A = (A_{s,t})_{s,t \in M\setminus\{z\}}
\]
defined by
\[
A_{s,t} =
\begin{cases}
c(s,t), & \text{if } st = z',\\
0, & \text{otherwise},
\end{cases}
\]
has nonzero determinant. 

Throughout this section, we assume that \(M\) has a unique annihilating element \(z'\) and that the associated matrix \(A\) has no zero rows or columns. Observe that this condition is automatically satisfied whenever \(\det A\neq0\). Whenever the stronger assumption \(\det A\neq0\) is required, it will be stated explicitly.

\begin{lem}\label{AnnDetOne}
For every $s \in M \setminus \{z,z'\}$, there exists $t \in M \setminus \{z,z'\}$ such that
\(
st = z'.
\)
\end{lem}

\begin{proof}
No row of $A$ is identically zero. Hence, for each $s \in M \setminus \{z,z'\}$, there exists $t \in M \setminus \{z,z'\}$ such that $A_{s,t} \neq 0$. By definition of $A$, this is equivalent to $st = z'$.
\end{proof}

The lemma shows that every non-annihilating element of $M$ can be sent to the annihilating element $z'$ by right multiplication. 

\begin{ddef}
With the above notation, assume that $\det A\neq0$. We say that the monoid $M$ is \emph{$A$-diagonal} if, in the submatrix of $A$ indexed by $M\setminus\{1,z'\}$, each row and each column contains exactly one nonzero entry.
The monoid $M$ is said to be \emph{$A$-upper triangular} if this submatrix can be transformed into an upper triangular matrix by suitable permutations of its rows and columns.
\end{ddef}

The advantage of $M$ being $A$-diagonal or $A$-upper triangular is that the determinant of $A$ depends only on the pattern of its zero and nonzero entries. In particular, its value is independent of the specific values of the cocycle $c(s,t)$; it is enough to determine whether or not $st=z'$.

Consequently, the properties of being $A$-diagonal or $A$-upper triangular apply equally to any matrix whose zero--nonzero pattern is determined solely by the products in the underlying monoid. In particular, they apply to the matrices $A(\chi)$ arising in Section~\ref{SemAlgebraComSem}, where
\[
A(\chi)_{ij}=
\begin{cases}
\chi(m_im_j), & \text{if } m_im_j\in Gm_{i_\chi},\\
0, & \text{otherwise},
\end{cases}
\qquad i,j\in J_\chi.
\]
Therefore, whenever our results imply that the determinant of such a matrix depends only on its zero--nonzero pattern, it is enough to verify that $\det A(\chi)\neq0$ for a single character (for example, the trivial character $1_G$). It then follows that $\det A(\chi)\neq0$ for every character $\chi$.

Note that in the matrix $A$, we have
\(A_{z',t}\neq 0\)
if and only if
\(t=1\). 
Moreover, for every $s \in M$, if there exists an element $t \in S$ such that $A_{s,t}\neq 0$, then $s\neq z'$, and consequently
\(A_{s,1}=0\).
For convenience, we omit the rows and columns indexed by the identity element and by $z'$, and work with the resulting submatrix of $A$. Accordingly, in the definitions of $A$-diagonal and $A$-upper triangular, as well as in the computation of $\det A$, we always refer to this submatrix.

For some particular cases, we describe the semigroup $S$ by means of the minimal identities satisfied by $B$, together with the identities determining the unique annihilating element of $M$. We restrict our attention to the cases $|B|=1$ and $|B|=2$, namely $B = \{r,s\}$.

First, suppose that $|B| = 1$ and $B = \{r\}$. Then
\(
S = \langle r \rangle
\)
is a commutative nilpotent semigroup. Let $k>1$ denote the nilpotency degree of $r$. Since $S$ is generated by $r$, the nilpotency degrees of $S$ and $r$ coincide. Therefore,
\[
r^k = z
\qquad \text{and} \qquad
r^{k-1} \neq z.
\]
In this case, the unique annihilating element of $M$ is
\(
z' = r^{k-1},
\)
and consequently, $M$ is $A$-diagonal.

Now suppose that $|B| = 2$ and $B = \{r,s\}$. 
Then $S = \langle r,s \rangle$. 
Suppose that the nilpotency degrees of $r$ and $s$ are $1<k$ and $1<k'$, respectively.
Since $r^{k-1}, s^{k'-1} \in M \setminus \{z\}$, Lemma~\ref{AnnDetOne} implies that there exist integers $l_s, l_r \geq 0$ such that
\[
r^{k-1}s^{l_s} = z', \qquad r^{l_r}s^{k'-1} = z'.
\]

\begin{lem}\label{ii'j}
There do not exist integers $i$ and $j' < j$ such that
\[
r^i s^j = r^i s^{j'} = z'.
\]
Also, there do not exist integers $i' < i$ and $j$ such that
\[
r^i s^j = r^{i'} s^{j} = z'.
\]
\end{lem}

\begin{proof}
Suppose that
\(
r^i s^j = r^i s^{j'} = z'
\)
with $j' < j$. Then
\[
r^i s^j = r^i s^{j'} s^{j-j'} = z' s^{j-j'} = z,
\]
a contradiction.

Similarly, a contradiction is obtained in the other case.
\end{proof}


If $l_s = k'-1$, then
\(
r^{k-1}s^{k'-1} = z'.
\)
Then
\(
r^{l_r}s^{k'-1} = z' = r^{k-1}s^{k'-1},
\)
and hence by Lemma~\ref{ii'j}, $l_r = k-1$. The converse follows similarly.

\begin{lem}\label{k'-1k-1}
If $l_s = k'-1$ and $l_r = k-1$ then $M$ is $A$-diagonal.
\end{lem}

\begin{proof}
Let $x = r^i s^j \in M$, where $0 \le i \le k-1$ and $0 \le j \le k'-1$ with $x\neq 1,z'$. Suppose that
\(
x u = r^{k-1}s^{k'-1} = z'
\)
for some $u = r^{i'} s^{j'} \in M$.
As $xu= r^{i+i'} s^{j+j'}$, we have $i+i'\leq k-1$ and $j+j'\leq k'-1$ and, thus,
\begin{align*}
z'r^{k-1-(i+i')}s^{k'-1-(j+j')}&=xur^{k-1-(i+i')}s^{k'-1-(j+j')}\\&=r^{i+i'} s^{j+j'}r^{k-1-(i+i')}s^{k'-1-(j+j')}=z'.
\end{align*}
Hence, $i'=k-1-i$ and $j'=k'-1-j$. 
Therefore, $u$ is uniquely determined. It follows that for every $x \in M \setminus \{z,z',1\}$ there exists a unique $u \in M \setminus \{z,z',1\}$ such that $xu = z'$. Consequently, $M$ is $A$-diagonal.
\end{proof}

Throughout the remainder of this section, we assume that
\(l_s < k'-1\)
and
\(l_r < k-1\).

\begin{lem}\label{ijl1l2}
If
\(
r^{i}s^{j} = z',
\)
for some integers $i$ and $j$ with
\(i \neq l_r\) and \(j \neq l_s,\)
then
\[
l_r < i < k-1
\quad \text{and} \quad
l_s < j < k'-1.
\]
\end{lem}

\begin{proof}
By Lemma~\ref{ii'j}, $i\neq k-1$ and $j\neq k'-1$. Hence either $i<l_r$ or $l_r<i<k-1$, and similarly either $j<l_s$ or $l_s<j<k'-1$.

Suppose that $i<l_r$. Then
\[z'=r^{l_r}s^{k'-1}=r^{i}s^{j} r^{l_r-i}s^{k'-1-j}=z' r^{l_r-i}s^{k'-1-j}=z,\]
a contradiction. Hence $l_r<i<k-1$.

Similarly, 
we have \(
l_s<j<k'-1.
\)
\end{proof}

\begin{lem}\label{i1i2j1j2}
Suppose that
\[r^{i_1}s^{j_1}=r^{i_2}s^{j_2}=z'.\]

If
\(l_r<i_1<i_2<k-1,\)
then
\(l_s<j_2<j_1<k'-1.\)

Similarly, if
\(l_s<j_1<j_2<k'-1,\)
then
\(l_r<i_2<i_1<k-1.\)
\end{lem}

\begin{proof}
By Lemma~\ref{ii'j}, the integers $j_1,j_2$ are distinct.

Since
\(r^{i_1}s^{j_1}=r^{i_2}s^{j_2}=z'\)
and $i_1<i_2$, we cannot have $j_1<j_2$, for otherwise
\[r^{i_2}s^{j_2}=r^{i_1}s^{j_1}r^{i_2-i_1}s^{j_2-j_1}=z'r^{i_2-i_1}s^{j_2-j_1}=z,\]
a contradiction. Hence $j_1>j_2$.

Therefore, by Lemma~\ref{ijl1l2},
we obtain
\[l_s<j_2<j_1<k'-1.\]

The second statement follows by symmetry.
\end{proof}

By Lemma~\ref{i1i2j1j2}, there exist integers
\[
i_1,\ldots,i_n \quad\text{and}\quad j_1,\ldots,j_n,
\]
with $n\ge 2$, such that
\[
i_1=l_r<i_2<\cdots<i_{n-1}<i_n=k-1,
\]
\[
j_1=k'-1>j_2>\cdots>j_{n-1}>j_n=l_s,
\]
and
\begin{equation}\label{z'}
z'
 = r^{i_1}s^{j_1}
 = r^{i_2}s^{j_2}
 = \cdots
 = r^{i_{n-1}}s^{j_{n-1}}
 = r^{i_n}s^{j_n}.
\end{equation}
Moreover, if $i\notin\{i_1,\ldots,i_n\}$, then there does not exist an integer $j$ such that
\(z'=r^{i}s^{j}.\)
An analogous statement holds for the positions of the occurrences of $s$ in $z'$.

\begin{lem}\label{il1jl11}
Suppose that, for some $1\leq l\leq n-1$,
\(i_l+1<i_{l+1}\) and \(j_l>j_{l+1}+1\).
Then \(r^{i_l+1}s^{j_{l+1}+1}=0\).
\end{lem}

\begin{proof}
We prove, by descending induction on $t$, that
\(r^{i_{l+1}-t}s^j=0\)
for all integers $j>j_{l+1}$ and all integers $t$ satisfying
\(1\leq t\leq i_{l+1}-i_l-1\).

For $t=1$, let $j$ be the largest integer greater than $j_{l+1}$ such that
\(r^{i_{l+1}-1}s^j\neq 0\).
Then
\(r^{i_{l+1}-1}s^{j+1}=0\).
Moreover,
\(r^{i_{l+1}}s^j=0,\)
since $r^{i_{l+1}}s^{j_{l+1}}=z'$ and $j>j_{l+1}$.
Hence \(r^{i_{l+1}-1}s^j=z'\),
contradicting the fact that
\(i_l<i_{l+1}-1<i_{l+1}\).
Therefore
\(r^{i_{l+1}-1}s^j=0\)
for all $j>j_{l+1}$.

Now assume that, for some
\(1\leq t<i_{l+1}-i_l-1,\)
we have
\(r^{i_{l+1}-t}s^j=0\)
for all $j>j_{l+1}$.
Let $j$ be the largest integer greater than $j_{l+1}$ such that
\(r^{i_{l+1}-(t+1)}s^j\neq 0.\)
Then
\(r^{i_{l+1}-(t+1)}s^{j+1}=0,\)
and, by the induction hypothesis,
\(r^{i_{l+1}-t}s^j=0.\)
Hence
\(r^{i_{l+1}-(t+1)}s^j=z',\)
contradicting the fact that
\(i_l<i_{l+1}-(t+1)<i_{l+1}.\)
Therefore
\(r^{i_{l+1}-(t+1)}s^j=0\)
for all $j>j_{l+1}$.

By induction,
\(r^{i_l+1}s^j=0\)
for all $j>j_{l+1}$, and in particular
\(r^{i_l+1}s^{j_{l+1}+1}=0.\)
\end{proof}

As $S$ is nilpotent, the set
\[
\mathcal{B}:=\{\,r^is^j \mid r^is^j\neq z,1\,\}
\]
is finite. We therefore define the matrices associated with $\mathcal{B}$, denoted by $A_{\mathcal{B}}$ and $A^{(1)}_{\mathcal{B}}$, respectivly, in the same way as the matrix $A$ associated with $S$. Their rows and columns are indexed by the elements of $\mathcal{B}$, and
\[
A_{\mathcal{B}}=\bigl(a_{u,v}\bigr)_{u,v\in\mathcal{B}}, A^{(1)}_{\mathcal{B}}=\bigl(a'_{u,v}\bigr)_{u,v\in\mathcal{B}},
\]
where
\[
a_{r^is^j,\;r^{i'}s^{j'}}
=
\begin{cases}
r^{i+i'}s^{j+j'}, & \text{if } r^is^jr^{i'}s^{j'}=z'\ \text{in}\ S,\\[1mm]
0, & \text{otherwise}.
\end{cases}
\]
\[
a'_{r^is^j,\;r^{i'}s^{j'}}
=
\begin{cases}
1, & \text{if } r^is^jr^{i'}s^{j'}=z'\ \text{in}\ S,\\[1mm]
0, & \text{otherwise}.
\end{cases}
\]


Let $r^is^j\in\mathcal{B}$ be such that the corresponding row of $A_{\mathcal{B}}$ contains at least two nonzero entries.
Let
\[
r^{i^{(1)}}s^{j^{(1)}},\ldots,r^{i^{(m)}}s^{j^{(m)}}\in\mathcal{B}
\]
be precisely those column indices for which the common row has a nonzero entry. For each $h\in\{1,\ldots,m\}$, write
\(
r^is^j\,r^{i^{(h)}}s^{j^{(h)}}=r^{\alpha_h}s^{\beta_h}\) in $B^+$
where
\(
\alpha_h\in\{i_1,\ldots,i_n\}\) and \(\beta_h\in\{j_1,\ldots,j_n\}\).

Reordering the columns if necessary, we may assume that
\[
\alpha_1\leq\alpha_2\leq\cdots\leq\alpha_m.
\]

We have
\[
\begin{array}{c|cccc}
 & r^{i^{(1)}}s^{j^{(1)}} & r^{i^{(2)}}s^{j^{(2)}} & \cdots & r^{i^{(m)}}s^{j^{(m)}} \\ \hline
r^is^j
& r^{\alpha_1}s^{\beta_1}
& r^{\alpha_2}s^{\beta_2}
& \cdots
& r^{\alpha_m}s^{\beta_m}
\end{array}
\]
Since
\(
\alpha_1\leq\cdots\leq\alpha_m,
\)
it follows from Lemma~\ref{i1i2j1j2} that
\(
\beta_m\leq\cdots\leq\beta_1.
\)
Moreover, for each $h\in\{1,\ldots,m\}$,
\[
i+i^{(h)}=\alpha_h
\qquad\text{and}\qquad
j+j^{(h)}=\beta_h.
\]
Hence
\(
i^{(1)}\leq i^{(2)}\leq\cdots\leq i^{(m)}
\)
and
\(
j^{(m)}\leq j^{(m-1)}\leq\cdots\leq j^{(1)}.
\)
Since the columns
\(
r^{i^{(1)}}s^{j^{(1)}},\ldots,r^{i^{(m)}}s^{j^{(m)}}
\)
are assumed to be pairwise distinct, all of the above inequalities are in fact strict. Hence
\[
\alpha_1<\alpha_2<\cdots<\alpha_m,
\qquad
\beta_m<\beta_{m-1}<\cdots<\beta_1,
\]
and
\[
i^{(1)}<i^{(2)}<\cdots<i^{(m)},
\qquad
j^{(m)}<j^{(m-1)}<\cdots<j^{(1)}.
\]

Suppose that there exist an integer \(\alpha\in\{i_1,\ldots,i_n\}\) and an index \(1\leq t\leq m-1\) such that
\(\alpha_t<\alpha<\alpha_{t+1}\).
Let \(\beta\in\{j_1,\ldots,j_n\}\)
be the unique integer such that
\(z'=r^{\alpha}s^{\beta}\) and thus \(\beta_t>\beta>\beta_{t+1}\).

Since
\[r^is^j\,r^{\,i^{(t)}+(\alpha-\alpha_t)}s^{\,j^{(t)}+(\beta-\beta_t)}
=
r^{\alpha}s^{\beta},\]
the entry in the row indexed by $r^is^j$ and the column indexed by
\[
r^{\,i^{(t)}+(\alpha-\alpha_t)}s^{\,j^{(t)}+(\beta-\beta_t)}
\]
is nonzero. However, \(\alpha_t<\alpha<\alpha_{t+1}\), and hence
\(i^{(t)}<i^{(t)}+(\alpha-\alpha_t)<i^{(t+1)}\),
which contradicts the fact that
\(r^{i^{(1)}}s^{j^{(1)}},\ldots,r^{i^{(m)}}s^{j^{(m)}}\)
are precisely the column indices corresponding to the nonzero entries of the row indexed by $r^is^j$.

Therefore, no such $\alpha$ exists. By symmetry, no such $\beta$ exists. 
It follows that there exists an integer $l$ such that
\[
(\alpha_1,\ldots,\alpha_m)
=
(i_l,\ldots,i_{l+m-1})
\]
and
\[
(\beta_1,\ldots,\beta_m)
=
(j_l,\ldots,j_{l+m-1}).
\]

The following lemma follows readily from the preceding discussion and the description of the annihilator elements in~(\ref{z'}) and will be used to determine the number of nonzero entries in the rows of $A_{\mathcal B}$.

Define
\[
i_0=
\begin{cases}
0, & \text{if } i_1\neq0,\\
i_1, & \text{if } i_1=0,
\end{cases}
\qquad
j_{n+1}=
\begin{cases}
0, & \text{if } j_n\neq0,\\
j_n, & \text{if } j_n=0.
\end{cases}
\]

\begin{lem}\label{ijNonZeroEntry}
Let $r^is^j$ index a row of $A_{\mathcal B}$.

\begin{enumerate}
\item If $(i,j)=(i_\alpha,j_\alpha)$ for some $1\leq\alpha\leq n$, then the corresponding row has no nonzero entries.

\item If $i=i_n$ with $j_n\neq 0$, then the corresponding row has exactly one nonzero entry. This occurs only when $0\leq j<l_s$.

\item If $j=j_1$ with $i_1\neq 0$, then the corresponding row has exactly one nonzero entry. This occurs only when $0\leq i<l_r$.

\item Assume that $i\neq i_n$ and $j\neq j_1$. 
Choose integers $0\leq\alpha<n$ and $0<\beta\leq n$ such that
\[
i_\alpha\leq i<i_{\alpha+1}
\qquad\text{and}\qquad
j_{\beta+1}\leq j<j_\beta.
\]
Then:
\begin{enumerate}
\item If $i=i_\alpha$ and $j=j_{\beta+1}$, then necessarily $\alpha\leq\beta$, and the row indexed by $r^is^j$ has $\beta-\alpha+2$ nonzero entries.

\item If exactly one of the equalities $i=i_\alpha$ and $j=j_{\beta+1}$ holds, then necessarily $\alpha\leq\beta$, and the row indexed by $r^is^j$ has $\beta-\alpha+1$ nonzero entries.

\item If $i\neq i_\alpha$ and $j\neq j_{\beta+1}$, then necessarily $\alpha<\beta$, and the row indexed by $r^is^j$ has $\beta-\alpha$ nonzero entries.
\end{enumerate}
\end{enumerate}
\end{lem}

Since the semigroup $S$ is generated by $B$, there is a natural surjective homomorphism from $B^{+}$ onto $S$. Thus, the matrix $A_{\mathcal B}$ associated with $B^{+}$ can be used to derive properties of the matrix $A$ associated with $S$.
The following lemma illustrates this approach. For example, if two rows of $A_{\mathcal B}$ each have exactly one nonzero entry, occurring in the same column, then the corresponding elements of $B^{+}$ must represent the same element of $S$; otherwise, the corresponding rows of $A$ would each have exactly one nonzero entry in the same column, forcing $\det A=0$. More generally, if a linear dependence among rows of $A_{\mathcal B}$ is preserved under the natural homomorphism from $B^{+}$ onto $S$, then the corresponding rows of $A$ remain linearly dependent. Consequently, $\det A=0$, a contradiction.

\begin{lem}\label{n4}
Suppose that there exists an integer $1\leq \alpha<n$ such that the following conditions hold:
\begin{enumerate}
\item either
\(i_{\alpha+1}-i_{\alpha}\leq i_1,\)
or there exists an integer $1\leq \alpha_1<n$ such that
\(i_{\alpha+1}-i_{\alpha}<i_{\alpha_1+1}-i_{\alpha_1};\)

\item either
\(j_{\alpha}-j_{\alpha+1}\leq j_n,\)
or there exists an integer $1\leq \alpha_2<n$ such that
\(j_{\alpha}-j_{\alpha+1}<j_{\alpha_2}-j_{\alpha_2+1}.\)
\end{enumerate}

Then $\det A=0$.
\end{lem}

\begin{proof}
We consider the four possible cases separately.
\begin{enumerate}
\item $i_{\alpha+1}-i_\alpha\leq i_1$ and $j_\alpha-j_{\alpha+1}\leq j_n$.

Then, $i_1,j_n\neq 0$ and thus $i_\alpha, j_{\alpha+1}\neq 0$.
In this case, the rows indexed by
\[
r^{i_n}s^{j_n-j_\alpha+j_{\alpha+1}},
\qquad
r^{i_\alpha}s^{j_{\alpha+1}},
\qquad
r^{i_1-i_{\alpha+1}+i_\alpha}s^{j_1}
\]
contain the submatrix
\[
\begin{array}{c|cc}
& s^{j_\alpha-j_{\alpha+1}} & r^{i_{\alpha+1}-i_\alpha} \\ \hline
r^{i_n}s^{j_n-j_\alpha+j_{\alpha+1}}
& r^{i_n}s^{j_n} & 0 \\
r^{i_\alpha}s^{j_{\alpha+1}}
& r^{i_\alpha}s^{j_\alpha}
& r^{i_{\alpha+1}}s^{j_{\alpha+1}} \\
r^{i_1-i_{\alpha+1}+i_\alpha}s^{j_1}
& 0
& r^{i_1}s^{j_1}
\end{array}
\]
of $A_{\mathcal B}$, since none of its rows is indexed by the identity element.
By Lemma~\ref{ijNonZeroEntry}, the first and third rows each have exactly one nonzero entry, whereas the second row has exactly two nonzero entries in the matrix $A_{\mathcal B}$. Hence these rows are linearly dependentin the matrix $A$, and therefore $\det A=0$.

\item $i_{\alpha+1}-i_\alpha\leq i_1$ and there exists an integer $1\leq\alpha_2<n$ such that
\[
j_\alpha-j_{\alpha+1}
<
j_{\alpha_2}-j_{\alpha_2+1}.
\]

Then, $i_1\neq 0$ and thus $i_\alpha\neq 0$.
In this case, the rows indexed by
\[
r^{i_{\alpha_2}}s^{j_{\alpha_2}-j_\alpha+j_{\alpha+1}},
\qquad
r^{i_\alpha}s^{j_{\alpha+1}},
\qquad
r^{i_1-i_{\alpha+1}+i_\alpha}s^{j_1}
\]
contain the submatrix
\[
\begin{array}{c|cc}
& s^{j_\alpha-j_{\alpha+1}} & r^{i_{\alpha+1}-i_\alpha} \\ \hline
r^{i_{\alpha_2}}s^{j_{\alpha_2}-j_\alpha+j_{\alpha+1}}
& r^{i_{\alpha_2}}s^{j_{\alpha_2}} & 0 \\
r^{i_\alpha}s^{j_{\alpha+1}}
& r^{i_\alpha}s^{j_\alpha}
& r^{i_{\alpha+1}}s^{j_{\alpha+1}} \\
r^{i_1-i_{\alpha+1}+i_\alpha}s^{j_1}
& 0
& r^{i_1}s^{j_1}
\end{array}.
\]

Since
\[
j_{\alpha_2+1}
<
j_{\alpha_2}-j_\alpha+j_{\alpha+1}
<
j_{\alpha_2},
\]
Lemma~\ref{ijNonZeroEntry} implies that the first row has exactly one nonzero entry in the matrix $A_{\mathcal B}$. Similarly, the third row has exactly one nonzero entry, while the second row has exactly two nonzero entries. Hence these rows are linearly dependent, and therefore $\det A=0$.

\item $j_\alpha-j_{\alpha+1}\leq j_n$ and there exists an integer $1\leq \alpha_1<n$ such that
\(i_{\alpha+1}-i_{\alpha}<i_{\alpha_1+1}-i_{\alpha_1}.\)

The proof is analogous to that of {\rm(2)}, replacing the third row by
\[
r^{i_{\alpha_1+1}-i_{\alpha+1}+i_\alpha}s^{j_{\alpha_1+1}}.
\]
Since
\[
i_{\alpha_1}
<
i_{\alpha_1+1}-i_{\alpha+1}+i_\alpha
<
i_{\alpha_1+1},
\]
Lemma~\ref{ijNonZeroEntry} shows that the third row has exactly one nonzero entry  in the matrix $A_{\mathcal B}$. Therefore $\det A=0$.

\item There exist integers $1\leq\alpha_1,\alpha_2<n$ such that
\[
i_{\alpha+1}-i_\alpha<i_{\alpha_1+1}-i_{\alpha_1}
\quad\text{and}\quad
j_\alpha-j_{\alpha+1}<j_{\alpha_2}-j_{\alpha_2+1}.
\]

By combining the arguments of {\rm(2)} and {\rm(3)}, we see that the first and third rows each have exactly one nonzero entry, while the second row has exactly two nonzero entries, provided that none of them is indexed by the identity element. The corresponding submatrix of $A_{\mathcal B}$ is
\[
\begin{array}{c|cc}
 & s^{j_\alpha-j_{\alpha+1}} & r^{i_{\alpha+1}-i_\alpha} \\ \hline
r^{i_{\alpha_2}}s^{j_{\alpha_2}-j_\alpha+j_{\alpha+1}}
& r^{i_{\alpha_2}}s^{j_{\alpha_2}} & 0 \\
r^{i_\alpha}s^{j_{\alpha+1}}
& r^{i_\alpha}s^{j_\alpha}
& r^{i_{\alpha+1}}s^{j_{\alpha+1}} \\
r^{i_{\alpha_1+1}-i_{\alpha+1}+i_\alpha}s^{j_{\alpha_1+1}}
& 0
& r^{i_{\alpha_1+1}}s^{j_{\alpha_1+1}}
\end{array}
\]

It remains to verify that these rows are not indexed by the identity element. Since
\[
j_{\alpha_2}-j_\alpha+j_{\alpha+1}>0
\quad\text{and}\quad
i_{\alpha_1+1}-i_{\alpha+1}+i_\alpha>0,
\]
the first and third rows are not indexed by the identity element. The second row is indexed by the identity element only if
\(n=2\) and \(\alpha=1.\)
However, in this case $\alpha_1=\alpha_2=1$, contradicting the assumptions.

Hence these rows are linearly dependent, and therefore $\det A=0$.
\end{enumerate}
\end{proof}

By Lemma~\ref{n4}, we obtain the following corollary under the assumption that $\det A\neq 0$.

\begin{cor}\label{n4cor}
Assume that $\det A\neq0$, and let
\[l^{(1)}=\max\{i_{\alpha+1}-i_\alpha\mid 1\leq\alpha<n\},
\qquad
l^{(2)}=\max\{j_\alpha-j_{\alpha+1}\mid 1\leq\alpha<n\}.\]
Then, for every $1\leq\alpha<n$, the following hold:
\begin{enumerate}
\item
If
\(i_{\alpha+1}-i_\alpha<l^{(1)}\)
or
\(i_{\alpha+1}-i_\alpha\leq i_1,\)
then
\(j_\alpha-j_{\alpha+1}=l^{(2)}\)
and
\(j_n<l^{(2)}.\)

\item
If
\(j_\alpha-j_{\alpha+1}<l^{(2)}\)
or
\(j_\alpha-j_{\alpha+1}\leq j_n,\)
then
\(i_{\alpha+1}-i_\alpha=l^{(1)}\)
and
\(i_1<l^{(1)}.\)
\end{enumerate}
\end{cor}

The previous corollary implies that if $j_n\geq l^{(2)}$, then
\(i_{\alpha+1}-i_\alpha=l^{(1)}\)
for every $1\leq\alpha<n$, and moreover $i_1<l^{(1)}$. Dually, if $i_1\geq l^{(1)}$, then
\(j_\alpha-j_{\alpha+1}=l^{(2)}\)
for every $1\leq\alpha<n$, and moreover $j_n<l^{(2)}$. Then,
\(j_n\geq l^{(2)}\) and \(i_1\geq l^{(1)}\)
cannot hold simultaneously.

\begin{lem}\label{n4-2}
Assume that $\det A\neq0$.
\begin{enumerate}
\item Fix an integer $0\leq\alpha\leq n-2$, and define
\[
i^{(\alpha)}
=
\begin{cases}
\min\{\,i_n-i_{n-1}-1,\; i_1\,\}, & \text{if }\alpha=0,\\[1mm]
\min\{\,i_n-i_{n-1}-1,\; i_{\alpha+1}-i_\alpha-1\,\}, & \text{if }1\leq\alpha\leq n-2,
\end{cases}
\]
and
\[
j^{(\alpha)}
=
\min\{\,j_n,\; j_{\alpha+1}-j_{\alpha+2}-1\,\}.
\]
Then
\[
r^{i_n-i^{(\alpha)}}s^{j_n-j^{(\alpha)}}
=
r^{i_{\alpha+1}-i^{(\alpha)}}s^{j_{\alpha+1}-j^{(\alpha)}}
\]
in $S$.

\item Fix $1\leq\alpha\leq n-1$, and define
\[
i_{(\alpha)}
=
\min\{ i_1,\; i_{\alpha+1}-i_{\alpha}-1\}.
\]\[
j_{(\alpha)}
=
\begin{cases}
\min\{ j_1-j_2-1,\; j_{n}\}, & \text{if }\alpha=n-1,\\[1mm]
\min\{ j_1-j_2-1,\; j_{\alpha+1}-j_{\alpha+2}-1\},& \text{if }1\leq\alpha\leq n-2,
\end{cases}\]
Then
\[
r^{i_1-i_{(\alpha)}}s^{j_1-j_{(\alpha)}}
=
r^{i_\alpha-i_{(\alpha)}}s^{j_\alpha-j_{(\alpha)}}
\] in $S$.

\item Fix distinct integers $0\leq\alpha,\beta\leq n-2$, and define
 \[
i_{\alpha,\beta}
=
\min\{\,i_{\alpha+1}-i_\alpha-1,\;
i_{\beta+1}-i_\beta-1\,\},
\]
\[
j_{\alpha,\beta}
=
\min\{\,j_{\alpha+1}-j_{\alpha+2}-1,\;
j_{\beta+1}-j_{\beta+2}-1\,\}.
\]
Then
\[
r^{i_{\alpha+1}-i_{\alpha,\beta}}
s^{j_{\alpha+1}-j_{\alpha,\beta}}
=
r^{i_{\beta+1}-i_{\alpha,\beta}}
s^{j_{\beta+1}-j_{\alpha,\beta}}
\]
in $S$.
\end{enumerate}
\end{lem}

\begin{proof}
The submatrix of $A_{\mathcal B}$ corresponding to the rows indexed by
\[
r^{i_n-i^{(\alpha)}}s^{j_n-j^{(\alpha)}}
\quad\text{and}\quad
r^{i_{\alpha+1}-i^{(\alpha)}}s^{j_{\alpha+1}-j^{(\alpha)}}
\]
is as follows:
\[
\begin{array}{c|cccccccc}
 &~ & \cdots &~& r^{i^{(\alpha)}}s^{j^{(\alpha)}} &~& \cdots &~ \\ \hline
r^{i_{n}-i^{(\alpha)}}s^{j_n-j^{(\alpha)}} & 0 & \cdots & 0 & r^{i_n}s^{j_n} & 0 & \cdots & 0 \\
r^{i_{\alpha+1}-i^{(\alpha)}}s^{j_{\alpha+1}-j^{(\alpha)}} & 0 & \cdots & 0 & r^{i_{\alpha+1}}s^{j_{\alpha+1}} & 0 & \cdots & 0 \\
\end{array}
\]
By Lemma~\ref{ijNonZeroEntry}, each of these rows has exactly one nonzero entry, occurring in the column indexed by $r^{i^{(\alpha)}}s^{j^{(\alpha)}}$. Since $\det A\neq0$, these rows cannot be distinct. Hence
\[
r^{i_n-i^{(\alpha)}}s^{j_n-j^{(\alpha)}}=r^{i_{\alpha+1}-i^{(\alpha)}}s^{j_{\alpha+1}-j^{(\alpha)}}
\]
in the semigroup $S$.

%

The proofs of {\rm(2)} and {\rm(3)} are analogous to that of {\rm(1)} and are therefore omitted.
\end{proof}

\begin{thm}\label{l_sAdd1L_2Add1}
Suppose that $\det A\neq0$ and $n=2$.
The matrix $A$ is $A$-diagonal or $A$-upper triangular.

Moreover, we distinguish the following cases, in each of which one additional identity holds:
\begin{enumerate}
\item[(i)] $2l_r+2-k \geq 0$ and $2l_s+2-k' \leq 0$:
\(
r^{l_r+1} = r^{2l_r+2-k}s^{k'-1-l_s}.
\)
\item[(ii)] $2l_r+2-k \leq 0$ and $2l_s+2-k' \geq 0$:
\(
s^{l_s+1} = r^{k-1-l_r}s^{2l_s+2-k'}.
\)
\item[(iii)] $2l_r+2-k \leq 0$ and $2l_s+2-k' \leq 0$:
\(
r^{k-1-l_r} = s^{k'-1-l_s}.
\)
\end{enumerate}
In each case, this additional identity together with the common identities
\begin{align*}
r^k = s^{k'} = z,\qquad
r^{l_r+1}s^{l_s+1} = z,\qquad
r^{k-1}s^{l_s} = r^{l_r}s^{k'-1} = z',
\end{align*}
completely determine the semigroup $S$.

The case
\[
2l_r+2-k > 0
\quad \text{and} \quad
2l_s+2-k' > 0
\]
cannot occur.
\end{thm}

\begin{proof}
By Lemma~\ref{k'-1k-1}, if $l_s=k'-1$ and $l_r=k-1$, then $M$ is $A$-diagonal. Hence, we may assume that $l_s<k'-1$ and $l_r<k-1$. Consequently,
\(k-2-l_r\geq0\) and \(k'-2-l_s\geq0\).

By Lemma~\ref{il1jl11}, we have $r^{l_r+1}s^{l_s+1} = z$.

As $n=2$, by Lemma~\ref{ijNonZeroEntry}, each row of the matrix $A$ contains at most two nonzero entries. Moreover, for a row $r^is^j$ whenever the exponent of $r$ exceeds $l_r$ or the exponent of $s$ exceeds $l_s$, the corresponding row contains at most one nonzero entry. Since $S$ is commutative, the same observations hold for the columns: each column contains at most two nonzero entries. Furthermore, the corresponding condition on the exponents of $r$ and $s$ applies analogously to the elements indexing the columns.

We apply Lemma~\ref{n4-2}(1) with $\alpha=0$ and $n=2$ (alternatively, Lemma~\ref{n4-2}(2) may be used). Then
\[i^{(0)}=\min\{\,i_2-i_1-1,\; i_1\,\}=\min\{\,k-2-l_r,\; l_r\,\},\]
and
\[j^{(0)}=\min\{\,j_2,\; j_1-j_2-1\,\}=\min\{\,l_s,\; k'-2-l_s\,\}.\]
Then, we distinguish four cases according to the signs of the integers
\[
2l_r+2-k
\quad\text{and}\quad
2l_s+2-k'.
\]
In each case, we apply Lemma~\ref{n4-2}(1) under the assumption that $\det A\neq0$ to obtain an identity in $S$. We then use this identity to derive the defining identities of $S$. Finally, we show that $A$ is equivalent to an upper triangular matrix.
\begin{enumerate}
\item[(i)] $(2l_r+2-k\geq 0)$ and $(2l_s+2-k'\leq 0)$:

In This case, we have $i^{(0)}=k-2-l_r$ and $j^{(0)}= l_s$.
Then, we obtain
\begin{equation}\label{eq1-1}
r^{l_r+1} = r^{2l_r+2-k}s^{k'-1-l_s}.
\end{equation}

Hence, $S$ satisfies the identities
\begin{equation}\label{eq1}
r^k = s^{k'} = z,
\qquad
r^{l_r+1}s^{l_s+1} = z,
\qquad
r^{k-1}s^{l_s} = r^{l_r}s^{k'-1} = z'.
\end{equation}
together with the identities above on $B$. We will show that these identities completely determine $S$, and that the corresponding matrix $A$ is equivalent to an upper triangular matrix.

Suppose that the distinct rows $r^{i_1}s^{j_1}$ and $r^{i_2}s^{j_2}$ both have nonzero entries in the column $r^{i_3}s^{j_3}$ in the matrix $A_{\mathcal B}$, without satisfying the identity~(\ref{eq1-1}).
That is, there do not exist integers $0\leq i'\leq k-2-l_r$ and $0\leq j'\leq l_s$ such that
\(i_1=i'+l_r+1\),
\(j_1=j'\),
\(i_2=i'+2l_r+2-k\),
and
\(j_2=j_1+k'-1-l_s\) or the same relations with the roles of $(i_1,j_1)$ and $(i_2,j_2)$ interchanged.
Since the two rows are distinct, their products with the column $r^{i_3}s^{j_3}$ are distinct in the matrix $A_{\mathcal B}$. Without loss of generality, assume that the product of the first row with the column is equal to $r^{k-1}s^{l_s}$, while the product of the second row with the column is equal to $r^{l_r}s^{k'-1}$.
First, suppose that neither of the two rows has a nonzero entry in any other column. If $i_3 > k-2-l_r$, then
\(i_1 < l_r+1\),
and hence the entry corresponding to this row and the column $r^{l_r-i_1}s^{k'-1-j_1}$ is nonzero, which is a contradiction. Therefore,
\(i_3 \leq k-2-l_r\). 
Taking
\(i' = k-2-l_r-i_3\),
it follows that these two rows together with the column arise from the identity~(\ref{eq1-1}).
Then at least one of the two rows has another nonzero entry. Since
\(j_2 = k'-1-j_3 = k'-1-l_s+j_1\)
and
\(l_s+1 \leq k'-1-l_s,\)
the row $r^{i_2}s^{j_2}$ cannot have any additional nonzero entries.
Now suppose that the row $r^{i_1}s^{j_1}$ has another nonzero entry corresponding to a column $r^{i_4}s^{j_4}$ distinct from $r^{i_3}s^{j_3}$. Then, the additional column must be
\(r^{l_r-i_1}s^{k'-1-j_1}\). This shows that these two rows are distinct in $S$.
\[\begin{array}{c|ccccccccc}
                      &~  & \cdots &~  & r^{i_3}s^{j_3}    &r^{i_4}s^{j_4} &~&~&~\\ \hline
r^{i_1}s^{j_1} & 0  & \cdots & 0 & r^{k-1}s^{l_s}  & r^{l_r}s^{k'-1}&0&\cdots&0 \\
r^{i_2}s^{j_2} & 0  & \cdots & 0 & r^{l_r}s^{k'-1}  &0                     &0&\cdots&0\\
r^{i_5}s^{j_5} & ~ & ~         & ~ & 0                      &0                     &~&~&~\\
\end{array}
\]

Now let $r^{i_5}s^{j_5}$ be a row distinct from both $r^{i_1}s^{j_1}$ and $r^{i_2}s^{j_2}$. Since the entries corresponding to the column $r^{i_3}s^{j_3}$ are already nonzero for the previous two rows, the entry of the row $r^{i_5}s^{j_5}$ in this column must be zero.
Suppose that the entry of the row $r^{i_5}s^{j_5}$ in the column $r^{i_4}s^{j_4}$ is nonzero. Then there are two possibilities: either this entry is equal to $r^{l_r}s^{k'-1}$, in which case we would obtain $r^{i_5}s^{j_5} = r^{i_1}s^{j_1}$ in $B^{+}$, which is impossible since the rows are distinct; or this entry is equal to $r^{k-1}s^{l_s}$.
Then we have
\(j_5 + k' - 1 - j_1 = l_s\).
Hence,
\(j_1 = k' - 1 - l_s + j_5 > l_s\),
which is a contradiction.

Therefore, after suitable permutations of its rows and columns, the matrix $A$ decomposes into blocks consisting either of single entries or of $2\times2$ blocks having exactly one zero entry on the diagonal. Consequently, $A$ can be transformed into an upper triangular matrix. Moreover, the identities~(\ref{eq1}) together with the identity~(\ref{eq1-1}) completely determine $S$.

\item[(ii)] $(2l_r+2-k> 0)$ and $(2l_s+2-k'> 0)$:

By Lemma~\ref{n4}, this case does not occur.

\item[(iii)] $(2l_r+2-k\leq 0)$ and $(2l_s+2-k'\geq 0)$:

Similarly to the first case, $S$ satisfies the identities
\[
\begin{aligned}
&s^{l_s+1} = r^{k-1-l_r}s^{2l_s+2-k'}, \qquad
r^k = s^{k'} = z, \qquad
r^{l_r+1}s^{l_s+1} = z, \\
&r^{k-1}s^{l_s} = r^{l_r}s^{k'-1} = z'.
\end{aligned}
\]
These identities are imposed on $B$ and completely determine $S$. Moreover, the corresponding matrix $A$ is equivalent to an upper triangular matrix.

\item[(iv)] $(2l_r+2-k\leq 0)$ and $(2l_s+2-k'\leq 0)$:

In This case, we have $i^{(0)}=l_r$ and $j^{(0)}= l_s$.
%
%
Since $\det A \neq 0$, we obtain
\begin{equation}\label{eq2-2}
r^{k-1-l_r} = s^{k'-1-l_s}.
\end{equation}
Hence, $S$ satisfies the identities~(\ref{eq1}), together with the identities above on $B$. We will show that these identities completely determine $S$, and that the corresponding matrix $A$ is equivalent to a diagonal matrix.

Suppose that the distinct rows $r^{i_1}s^{j_1}$ and $r^{i_2}s^{j_2}$ in the matrix $A_{\mathcal B}$ both have nonzero entries in the column $r^{i_3}s^{j_3}$.
Without loss of generality, assume that the product of the first row with the column is equal to $r^{k-1}s^{l_s}$, while the product of the second row with the column is equal to $r^{l_r}s^{k'-1}$.

As the product of $r^{i_2}s^{j_2}$ and $r^{i_3}s^{j_3}$ is $r^{l_r}s^{k'-1}$, $i_3 \leq l_r$, and therefore
\(
i_1 + i_3 \leq i_1 + l_r.
\)
Now, the product of $r^{i_1}s^{j_1}$ and $r^{i_3}s^{j_3}$ is $r^{k-1}s^{l_s}$. Thus,
\(
i_1 + i_3 = k-1,
\)
and consequently
\(
k-1-l_r \leq i_1.
\)
Since $2l_r+2-k \leq 0$, this implies
\(
l_r+1 \leq i_1.
\)
Similarly, we obtain
\(
l_s+1 \leq j_2.
\) This shows that all entries of these rows, except for the column $r^{i_3}s^{j_3}$, are zero. Hence,
\[
r^{i_1}s^{j_1} = r^{i_2}s^{j_2}.
\]
Then we have
\[
r^{k-1-l_r+i_2}s^{j_1} = r^{i_2}s^{k'-1-l_s+j_1},
\]
which shows that this equality arises from the identity~(\ref{eq2-2}).
\[\begin{array}{c|cccccccc}
                      &~  & \cdots &~  & r^{i_3}s^{j_3}     &~&~&~\\ \hline
r^{i_1}s^{j_1} & 0  & \cdots & 0 & r^{k-1}s^{l_s}  &0&\cdots&0 \\
r^{i_2}s^{j_2} & 0  & \cdots & 0 & r^{l_r}s^{k'-1}  &0&\cdots&0\\
\end{array}
\]
\end{enumerate}
\end{proof}

For example, in the case $l_r=k-2$ or $l_s=k'-2$, Theorem~\ref{l_sAdd1L_2Add1} allows us to determine when $\det A\neq0$ and to derive the identities defining the semigroup $S$.

For $n=3$, unlike the case $n=2$ treated in Theorem~\ref{l_sAdd1L_2Add1}, it may be necessary to know the actual entries of the matrix $A$, and not merely its zero--nonzero pattern, in order to determine whether $\det A\neq0$. The following example illustrates this phenomenon: the matrix $A$ is neither $A$-diagonal nor $A$-upper triangular, and its determinant depends on the values of its nonzero entries.

\begin{example}\label{Ex1}
Let $M$ be the commutative monoid with zero $z$, group of units $\{1,a\}$, generators $a,r,s$, and defining relations
\[
\begin{aligned}
&a^2=1,\qquad
r^5=s^5=z,\qquad
s^4=r^2s^2=ar^4,\qquad
r^2s=s^3,\qquad
r^3=rs^2,\\
&rs^3=r^3s=z.
\end{aligned}
\]
Put $z'=s^4$. Then
\[
M\setminus\{1,a,z\}
=
\{r,ar,s,as,r^2,ar^2,rs,ars,s^2,as^2,r^3,ar^3,r^2s,ar^2s,s^4,as^4\}.
\]

Let $1_G$ denote the trivial character of $G$, and let $\chi$ be the nontrivial character. Take
\[
1,\ r,\ s,\ r^2,\ rs,\ s^2,\ r^3,\ r^2s,\ s^4
\]
as representatives of the $G$-orbits on $M\setminus\{z\}$. Since
\[
\chi(z')=1
\quad\text{and}\quad
\chi(az')=-1,
\]
the matrices $A(1_G)$ and $A(\chi)$ are
\[
A(1_G)=
\begin{bmatrix}
0&0&0&0&0&0&0&0&1\\
0&0&0&0&0&0&1&0&0\\
0&0&0&0&0&0&0&1&0\\
0&0&0&1&0&1&0&0&0\\
0&0&0&0&1&0&0&0&0\\
0&0&0&1&0&1&0&0&0\\
0&1&0&0&0&0&0&0&0\\
0&0&1&0&0&0&0&0&0\\
1&0&0&0&0&0&0&0&0
\end{bmatrix},
\]\[
A(\chi)=
\begin{bmatrix}
0&0&0&0&0&0&0&0&1\\
0&0&0&0&0&0&-1&0&0\\
0&0&0&0&0&0&0&1&0\\
0&0&0&-1&0&1&0&0&0\\
0&0&0&0&1&0&0&0&0\\
0&0&0&1&0&1&0&0&0\\
0&-1&0&0&0&0&0&0&0\\
0&0&1&0&0&0&0&0&0\\
1&0&0&0&0&0&0&0&0
\end{bmatrix}.
\]
Then
\[\det A(1_G)=0,
\qquad
\det A(\chi)\neq0.\]
Hence the algebra
\(K_0(M_{1_G}/G,c_{1_G})\)
is not Frobenius, whereas
\(K_0(M_{\chi}/G,c_{\chi})\)
is Frobenius.
\end{example}

To determine whether a twisted contracted monoid algebra over a finite field $K$ is Frobenius, it is necessary that the trivial-character component
\[K_0(M_{1_G}/G,c_{1_G})\]
be Frobenius, where $G$ is the group of units of $M$.

When the associated matrix is \(A\)-diagonal or equivalent, up to permutations of its rows and columns, to an upper triangular matrix, its determinant is automatically nonzero, since every diagonal entry is a nonzero cocycle value. Thus the Frobenius property depends only on the multiplication structure of the monoid and is independent of the particular cocycle.
In the general case, one may first ignore the cocycle values by replacing every nonzero entry of \(A\) with \(1\). If the resulting matrix already has determinant zero, then no choice of cocycle can produce a Frobenius algebra. Hence such monoids may be discarded before any computation involving the cocycle.

From this point onward, we shall always assume that whenever two rows of $A^{(1)}_{\mathcal{B}}$ are equal, the corresponding elements of $B^+$ represent the same element of $M$; otherwise, $\det A=0$ for trivial character.
For example, suppose that there exist integers $2\le l<l'$ and $m\ge2$ such that
\[
l+m-1<l'+m-1\le n-1,
\]
and
\[
i_{l+t}-i_{l+t-1}=i_{l'+t}-i_{l'+t-1},
\qquad
j_{l+t-1}-j_{l+t}=j_{l'+t-1}-j_{l'+t},
\]
for every $1\le t\le m-1$. Then, to avoid two rows of $A$ having the same pattern of nonzero entries, the analogue of Lemma~\ref{n4-2} must hold. Since its proof is a straightforward repetition of the argument of Lemma~\ref{n4-2}, we omit the details.

Consequently,
\[
r^{\,i_l-i'}s^{\,j_{l+m-1}-j'}
=
r^{\,i_{l'}-i'}s^{\,j_{l'+m-1}-j'}
\]
in $S$, where
\[
i'
=
\min\{\,i_l-i_{l-1},\, i_{l'}-i_{l'-1}\,\}-1,
\]
and
\[
j'
=
\min\{\,j_{l+m-1}-j_{l+m},\, j_{l'+m-1}-j_{l'+m}\,\}-1.
\]

Similarly to Lemma~\ref{n4-2}, the boundary cases $l=1$ and $l'+m-1=n$ admit analogous identities. 

Theorem~\ref{l_sAdd1L_2Add1} provides a complete description of the case $n=2$. Although the same approach can be extended to $n\ge3$, the arguments require considerably more case distinctions and additional identities among the generators, and we do not pursue them here.


\section{Character-Average Weights on Non-Frobenius Semigroup Algebras}\label{SecWeight}

As in Section~\ref{Two-G-A}, let $S$ be a finite commutative nilpotent semigroup, let
\(M=S\cup\{1\}\)
be the monoid obtained by adjoining an identity element, and let
\(K_0(M,c)\)
be a twisted contracted monoid algebra over an admissible finite field $K$. 
Throughout this section, we assume that \(M\) has a unique annihilating element \(z'\), that the associated matrix \(A\) has no zero rows or columns.


Following Proposition~\ref{GenCharMc}, define the linear map
\[
\lambda:K_0(M,c)\longrightarrow K
\]
by
\[
\lambda(z')=1,
\qquad
\lambda(x)=0
\quad\text{for all }x\neq z',
\]
and let
\(\chi=\psi\circ\lambda,\)
where $\psi:K\to\mathbb{C}^{\times}$ is a nontrivial additive character satisfying $\psi(1)\neq1$.

Consider the character-average function
\[
\omega(x)
=
1-
\frac{1}{|K_0(M,c)^\times|}
\sum_{u\in K_0(M,c)^\times}
\chi(xu).
\]
Our aim is to determine the nonzero principal ideals of \(K_0(M,c)\) on which the character-average weight \(\omega\) satisfies the defining conditions of a homogeneous weight, even when \(K_0(M,c)\) is not a Frobenius algebra.

Although \(M\) has a unique annihilating element by assumption, the associated matrix \(A\) need not be nonsingular. Consequently, the results of Section~\ref{SemAlgebraComSem} do not apply directly, and a separate analysis of the function \(\omega\) is required.

We compute the character-average function of an arbitrary element
\[x=\alpha_11+\sum_{s\in S}\alpha_ss,\]
where \(\alpha_1,\alpha_s\in K\) for all \(s\in S\).

First, we determine the group of units of \(K_0(M,c)\).

Since \(S\) is nilpotent, the contracted twisted semigroup algebra \(K_0(S,c)\) is a nilpotent ideal of \(K_0(M,c)\), and
\[K_0(M,c)=K1\oplus K_0(S,c).\]
Hence every element of \(K_0(M,c)\) can be written uniquely as
\[a+x,\qquad a\in K,\quad x\in K_0(S,c).\]
Moreover, \(a+x\) is invertible if and only if \(a\neq0\). Therefore,
\[
K_0(M,c)^{\times}
=
\left\{
\sum_{s\in M\setminus\{z\}}\alpha_s s
\;\middle|\;
\alpha_1\neq0
\right\}.
\]
Consequently,
\(|K_0(M,c)^{\times}|=(|K|-1)|K|^{|S|}.\)

Let \(x=\alpha_11+\sum_{s\in S}\alpha_ss.\)

Let \(u=\beta_11+\sum_{t\in S}\beta_tt
\in K_0(M,c)^{\times}, \)
where \(\beta_1\neq0\). Then
\[
xu
=
\sum_{a,b\in M\setminus\{z\}}
\alpha_a\beta_b\,c(a,b)\,ab.
\]
Since \(\lambda\) vanishes on all basis elements except \(z'\),
\[
\lambda(xu)
=
\sum_{\substack{a,b\in M\setminus\{z\}\\ab=z'}}
\alpha_a\beta_b\,c(a,b).
\]
Hence
\[
\omega(x)
=
1-
\frac{1}{|K_0(M,c)^{\times}|}
\sum_{u\in K_0(M,c)^{\times}}
\psi\!\left(
\sum_{\substack{a,b\in M\setminus\{z\}\\ab=z'}}
\alpha_a\beta_b\,c(a,b)
\right).
\]

Writing
\[
P(\beta)
=
\alpha_1\beta_{z'}
+
\sum_{\substack{a,b\in S\\ab=z'}}
\alpha_a\beta_b\,c(a,b),
\]
we obtain
\(\lambda(xu)=\alpha_{z'}\beta_1+P(\beta).\)

Suppose first that \(P(\beta)\) is identically zero. Then
\(\lambda(xu)=\alpha_{z'}\beta_1,\)
and therefore
\[
\sum_{u\in K_0(M,c)^{\times}}\chi(xu)
=
|K|^{|S|}
\sum_{a\in K^\times}\psi(\alpha_{z'}a).
\]

If \(\alpha_{z'}=0\), as in the case \(x=0\), then
\[
\sum_{u\in K_0(M,c)^{\times}}\chi(xu)
=
(|K|-1)|K|^{|S|},
\]
and hence
\(\omega(x)=0.\)

Assume now that \(\alpha_{z'}\neq0\), for example when \(x=kz'\) with \(k\in K^\times\). Since multiplication by the nonzero scalar \(\alpha_{z'}\) is a permutation of \(K^\times\),
\[
\sum_{a\in K^\times}\psi(\alpha_{z'}a)
=
\sum_{a\in K^\times}\psi(a)
=
-1.
\]
Consequently,
\[
\sum_{u\in K_0(M,c)^{\times}}\chi(xu)
=
-|K|^{|S|},
\]
and therefore
\[
\omega(x)
=
1+\frac{1}{|K|-1}=\frac{|K|}{|K|-1}.
\]

Now suppose that \(P(\beta)\) is not identically zero. After collecting like terms, at least one variable \(\beta_t\), where \(t\in S\), has a nonzero coefficient. Since \(\beta_t\) ranges over all elements of \(K\),
\(\sum_{\beta_t\in K}\psi(c\beta_t)=0,
\ (c\neq0).\)
Hence
\[
\sum_{u\in K_0(M,c)^{\times}}\chi(xu)=0,
\]
and consequently
\(\omega(x)=1.\)
%
%
%

The values of the character-average function obtained above are summarized in the following proposition.

\begin{prop}\label{PropCharacterAverage}
For
\(x=\alpha_11+\sum_{s\in S}\alpha_ss,\)
define
\[
P(\beta)
=
\alpha_1\beta_{z'}
+
\sum_{\substack{a,b\in S\\ab=z'}}
\alpha_a\beta_b\,c(a,b).
\]
Then the character-average function
\[
\omega(x)
=
1-
\frac{1}{|K_0(M,c)^{\times}|}
\sum_{u\in K_0(M,c)^{\times}}
\chi(xu)
\]
satisfies
\[
\omega(x)=
\begin{cases}
0,
&
\text{if }P(\beta)\equiv0\text{ and }\alpha_{z'}=0,\\[1ex]
\displaystyle
\frac{|K|}{|K|-1},
&
\text{if }P(\beta)\equiv0\text{ and }\alpha_{z'}\neq0,\\[2ex]
1,
&
\text{if }P(\beta)\not\equiv0.
\end{cases}
\]
\end{prop}

We now verify that the character-average function \(\omega\) satisfies the first and second conditions of a homogeneous weight given in Definition~\ref{homogeneous-weight}.

\begin{prop}\label{PropFirstTwoConditions}
The character-average function \(\omega\) satisfies conditions {\rm(i)} and {\rm(ii)} of Definition~\ref{homogeneous-weight}.
\end{prop}

\begin{proof}
By Proposition~\ref{PropCharacterAverage},
\(\omega(0)=0.\)
Hence condition {\rm(i)} is satisfied.

Now suppose that \(Rx=Ry\) for some elements \(x,y\in R\), where
\(R=K_0(M,c)\).
Then there exist \(a,b\in R\) such that
\[x=ay\quad\text{and}\quad y=bx.\]
Consequently,
\(x=abx\), and therefore
\((1-ab)x=0\).

Suppose that \(ab\) is not a unit. Since \(R\) is a local ring with maximal ideal
\(K_0(S,c)\), it follows that \(ab\in K_0(S,c)\). Hence \(1-ab\) is a unit. Multiplying the above equality by \((1-ab)^{-1}\), we obtain \(x=0\), and, thus \(y=0\). 
If \(x,y\neq 0\), then \(ab\) is a unit.

If \(a\) is not a unit, then \(a\in K_0(S,c)\). Since \(K_0(S,c)\) is an ideal of \(R\), then \(ab\in K_0(S,c)\), contradicting the fact that \(ab\) is a unit. Hence \(a\) is a unit. Therefore,
\(x=uy\)
for some unit \(u\in R^\times\).

It follows that
\[\omega(x)=\omega(uy)=1-\frac{1}{|R^\times|}\sum_{v\in R^\times}\chi(uyv).\]
Since multiplication by \(u\) is a bijection of \(R^\times\), the substitution
\(w=uv\)
yields
\[\sum_{v\in R^\times}\chi(uyv)=\sum_{w\in R^\times}\chi(yw).\]
%
%
%
%
Hence,
\[\omega(x)=1-\frac{1}{|R^\times|}\sum_{w\in R^\times}\chi(yw)=\omega(y).\]
Therefore condition {\rm(ii)} is also satisfied.
\end{proof}

Before proceeding further, we investigate condition \rm(iii) without assuming that every nonzero principal ideal contains \(Kz'\). The next two propositions show that the behavior of the character-average weight depends on the principal ideal under consideration. In particular, condition \rm(iii) holds for principal ideals containing \(Kz'\), while it fails for principal ideals contained in \(\ker(\lambda)\).

\begin{prop}\label{SingularMatrixNotHomogeneous}
Suppose that the associated matrix \(A\) is singular. Then there exists a
nonzero element \(x\in R\) such that
\(z'\notin Rx\)
and
\[
\omega(y)=0
\qquad
\text{for every }y\in Rx
\]
where \(R=K_0(M,c)\).
Consequently, \(\omega\) does not satisfy condition {\rm(iii)} in the
definition of a homogeneous weight.
\end{prop}

\begin{proof}
Since \(A\) is singular, \(R\) is not a Frobenius algebra. 
Hence, \(\ker(\lambda)\) contains a nonzero left ideal.
Therefore, there exists a nonzero element \(x\in R\) such that
\(Rx\subseteq\ker(\lambda).\)

We first show that
\(z'\notin Rx.\)
Indeed, if \(z'\in Rx\), then there exists \(a\in R\) such that
\(ax=z'.\)
It follows that
\[
0=\lambda(ax)=\lambda(z')=1,
\]
which is a contradiction.

Now let \(y\in Rx\). Then \(y=bx\) for some \(b\in R\). For every
\(u\in R^\times\), we have
\[\lambda(yu)=\lambda(bxu)=\lambda\bigl(x(bu)\bigr)=0.\]
Therefore,
\(\chi(yu)=1\) for every \(u\in R^\times,\)
and hence
\[
\omega(y)
=
1-\frac{1}{|R^\times|}
\sum_{u\in R^\times}\chi(yu)
=
0.
\]
Thus,
\[
\sum_{y\in Rx}\omega(y)=0.
\]

Since \(x\neq0\), the principal ideal \(Rx\) is nonzero. Therefore,
\[
\sum_{y\in Rx}\omega(y)\neq |Rx|,
\]
and condition {\rm(iii)} does not hold.
\end{proof}

\begin{prop}\label{PrincipalIdealAverage}
Assume that
\[
P(\beta)\equiv0
\quad\Longleftrightarrow\quad
x\in Kz'.
\]
Let \(x\in R\setminus\{0\}\) where \(R=K_0(M,c)\).

\begin{enumerate}
\item[\rm(i)]
If \(z'\in Rx\), then
\[
\sum_{y\in Rx}\omega(y)=|Rx|.
\]

\item[\rm(ii)]
If
\(Rx\subseteq\ker(\lambda),\)
then
\[
\sum_{y\in Rx}\omega(y)=0.
\]
Consequently, condition {\rm(iii)} fails on \(Rx\).
\end{enumerate}
\end{prop}

\begin{proof}
Suppose first that \(z'\in Rx\). Since \(Rx\) is a \(K\)-subspace of \(R\),
we have
\(Kz'\subseteq Rx.\)
If \(Rx=Kz'\), then every nonzero element of \(Rx\) has weight
\[
\frac{|K|}{|K|-1},
\]
and hence
\[
\sum_{y\in Rx}\omega(y)
=
(|K|-1)\frac{|K|}{|K|-1}
=
|K|
=
|Rx|.
\]

Suppose that \(Rx\neq Kz'\). By
Proposition~\ref{PropCharacterAverage}, every nonzero element of \(Kz'\)
has weight
\[
\frac{|K|}{|K|-1},
\]
whereas every element of \(Rx\setminus Kz'\) has weight \(1\).
Therefore,
\[
\begin{aligned}
\sum_{y\in Rx}\omega(y)
&=
(|K|-1)\frac{|K|}{|K|-1}
+
\bigl(|Rx|-|K|\bigr)\\
&=
|Rx|.
\end{aligned}
\]

Now, suppose that
\(Rx\subseteq\ker(\lambda).\)
By Proposition~\ref{SingularMatrixNotHomogeneous},
\[
\sum_{y\in Rx}\omega(y)=0.
\]
Since \(Rx\neq\{0\}\), we have
\[
\sum_{y\in Rx}\omega(y)\neq |Rx|.
\]
Hence condition {\rm(iii)} does not hold for the principal ideal \(Rx\).
\end{proof}

By Proposition~\ref{SingularMatrixNotHomogeneous}, applied to the algebra
\(K_0(M_{1_G}/G,c_{1_G})\) in Example~\ref{Ex1}, we have
\[
\sum_{y\in R(r^2-s^2)}\omega(y)=0.
\]

The preceding results show that, in general, the character-average weight need not satisfy condition {\rm(iii)} on every principal ideal. Nevertheless, it does satisfy this condition on a distinguished family of principal ideals, namely those containing the minimal ideal \(Kz'\). This motivates the following relative version of the notion of a homogeneous weight, applicable even when the algebra is not Frobenius.

\begin{ddef}\label{RelativeHomogeneousWeight}
Let \(R\) be a finite \(K\)-algebra, and let \(\mathcal{I}\) be a family of
nonzero principal left ideals of \(R\). A function
\[
\omega:R\longrightarrow\mathbb{R}
\]
is said to be homogeneous on \(\mathcal{I}\) if it satisfies conditions
{\rm(i)} and {\rm(ii)} of Definition~\ref{homogeneous-weight} and, for every
\(I\in\mathcal{I}\),
\[
\sum_{y\in I}\omega(y)=|I|.
\]
\end{ddef}

We now summarize the results established in this section. Under the assumption
\[
P(\beta)\equiv0
\quad\Longleftrightarrow\quad
x\in Kz',
\]
the character-average construction defines a homogeneous weight on the family of principal ideals containing \(Kz'\). In the Frobenius case, every nonzero principal ideal contains \(Kz'\), so this notion coincides with the classical definition of a homogeneous weight.

\begin{thm}\label{HomogeneousOnAdmissibleIdeals}
Let \(R=K_0(M,c)\), where \(M\) has a unique annihilating element \(z'\), and let
\[
\lambda:R\longrightarrow K
\]
be the \(K\)-linear map defined by
\[
\lambda(z')=1,
\qquad
\lambda(t)=0
\quad\text{for every basis element }t\neq z'.
\]
Let
\(\chi=\psi\circ\lambda\)
and define
\[
\omega(x)
=
1-\frac{1}{|R^\times|}
\sum_{u\in R^\times}\!\chi(xu).
\]
Assume that
\[
P(\beta)\equiv0
\quad\Longleftrightarrow\quad
x\in Kz'.
\]
Then, for every nonzero element \(x\in R\),
\[
\sum_{y\in Rx}\!\omega(y)=|Rx|
\quad\Longleftrightarrow\quad
Kz'\subseteq Rx.
\]
Consequently, the family
\[
\mathcal{I}_{z'}
=
\left\{
Rx\;\middle|\;
x\in R\setminus\{0\}
\text{ and }
Kz'\subseteq Rx
\right\}
\]
is precisely the family of nonzero principal ideals on which \(\omega\)
satisfies condition {\rm(iii)} of
Definition~\ref{homogeneous-weight}. In particular, \(\omega\) is homogeneous
on \(\mathcal{I}_{z'}\).
\end{thm}



\section*{Acknowledgments}
The author acknowledges the use of ChatGPT (OpenAI) for assistance with English-language editing.



\bibliographystyle{plain}
\bibliography{ref-Det}

\end{document}